\documentclass[11pt]{article}

\usepackage[a4paper]{geometry}
\usepackage{graphicx}
\usepackage[svgnames]{xcolor} 
 
\usepackage[utf8]{inputenc}
\usepackage[T1]{fontenc}

\usepackage{amsmath}
\usepackage{amssymb}
\usepackage{amsthm}
\usepackage{thmtools}
\usepackage{lscape}
\usepackage{tikz}

\usepackage{parskip}

\usepackage{bussproofs}

\declaretheoremstyle[
	spaceabove=6pt, spacebelow=6pt,
	headfont=\normalfont\bfseries,
	notefont=\mdseries, notebraces={\kern-0.7ex \textbf{:} }{},
	bodyfont=\normalfont,
]{theoremstyle}

\declaretheoremstyle[
	spaceabove=6pt, spacebelow=6pt,
	headfont=\normalfont\bfseries,
	notefont=\mdseries, notebraces={\kern-0.7ex \textbf{:} }{},
	bodyfont=\normalfont,
	qed=$\blacksquare$
]{definitionstyle}

\declaretheorem[style=theoremstyle]{theorem}
\declaretheorem{definition}[
style=definitionstyle,
sibling=theorem
]

\declaretheorem{lemma, proposition, corollary}[
style=theoremstyle,
sibling=theorem
]

\author{Borja Sierra Miranda \and Thomas Studer}
\date{}
\title{Cut elimination for a non-wellfounded system for the master modality\footnote{Research supported by the Swiss National Science Foundation project 200021\_214820.}}

\newcommand\nec{\ensuremath\Box}
\newcommand\necm{\ensuremath\Box^+}
\newcommand\dnecm{\ensuremath\boxdot^+}
\newcommand\NN{\mathbb{N}}

\usepackage{biblatex} 
\begin{document}

\maketitle

\begin{abstract} 
In~\cite{previous}, we provided a method for eliminating cuts in
non-wellfounded proofs with a local-progress condition, these being the simplest
kind of non-wellfounded proofs.
The method consisted of splitting the proof into nicely behaved fragments.
This paper extends our method to proofs based on simple trace conditions. 
The main idea is to split the system with  the trace condition into infinitely
many local-progress calculi that together are equivalent to the original
trace-based system. 
This provides a cut elimination method using only basic tools of structural
proof theory and corecursion, which is needed due to the non-wellfounded
character of proofs.
We will employ the method to obtain syntactic cut elimination for \(K^+\), a system
of modal logic with the master modality.
\end{abstract} 
\section*{Introduction}

Cut elimination in non-wellfounded/cyclic proof theory is currently an active topic of research and has been previously addressed by other researchers with different tools.
For example 
via finite approximations in \cite{acclavio2024infinitarycuteliminationfiniteapproximations},
via cyclic proofs in \cite{cut-cyclic},
via multicuts in \cite{afshari2024demystifyingmu},
via runs in \cite{das_et_al:LIPIcs.CSL.2018.19},
via infinitary rewriting in \cite{10.1007/978-3-031-43513-3_12},
via ultrametric spaces in \cite{ShamkanovwGrz}, \cite{ShamkanovGrz} and \cite{shamkanov2023structuralprooftheorymodal},
among others.
First proof-theoretic results on establishing weakening admissibility in cyclic calculi for linear temporal logic can be found in~\cite{KokkinisStuder+2016+171+192}.
The richness of methods is a witness of the hardness of the problem.
We believe this hardness arises, among other things, from two fundamental facts:
\begin{enumerate}
  \item In principle non-wellfounded proofs do not have a notion of height associated to them.
  This means that we cannot do recursion over them.
  \item Verifying that a tree is a proof in the finitary setting just requires to check a ``local'' condition globally, in particular that every node is the conclusion of a rule instance.
  In the non-wellfounded setting a global condition on the branches is added, proving that this condition is preserved after the process of cut elimination is the principal headache. 
\end{enumerate}

In \cite{previous} we defined a method to establish cut elimination and,  more generally,  to make proof translations between local progress sequent calculi.
These are sequent calculi allowing non-wellfounded proofs such that any infinite branch goes through some rules infinitely many times.
The method is based in splitting the proof in two parts, one of them amenable to tools of finitary proof theory such as recursion.
The other part needs to be treated corecursively, but this treatment is generic and does not need to be reimplemented for each sequent calculus.
The combination of these two facts brings really close the study of local progress calculi and finitary calculi, simplifying the proof theory of the first.

We left open the question whether this simplication can be adapted to other non-wellfounded systems with more complex branch conditions.
In this paper we answer this question positively providing a novel cut elimination method for a non-wellfounded system for the master modality (\(\text{K}^+\)).
Cut elimination for \(\text{K}^+\) was already proved by Shamkanov
in~\cite{shamkanov2023structuralprooftheorymodal} by the method of continuous
cut elimination via ultrametric spaces.
We replace the use of ultrametric spaces by partitions, a simplification which will make it possible to study further levels of the recursive-corecursive iteration scheme occuring naturally in non-wellfounded proof theory.

\setcounter{section}{-1}
\section{Preliminaries}
\subsection*{Master modality}
We use this section to introduce some concepts that are needed for the rest of the paper.

\(\text{K}^+\) is a modal logic, we will work with its formulation in the modal language with connectives \(\bot, \to\) and unary modalities \(\nec, \necm\).
The operator \(\necm\) is what we call the master modality.
We use usual Kripke models for the semantics of this logic.
If \(\mathbf{M} = (W,R, V)\) is a Kripke model, then the semantics for box and master formulas is as follows:
\begin{align*}
  &\mathbf{M}, w \vDash \nec \phi \text{ iff for any \(v\) such that \(wRv\) we
	have that \(\mathbf{M}, v \vDash \phi\)}, \\
  &\mathbf{M}, w \vDash \necm \phi \text{ iff for any \(v\) such that \(wR^+ v\)
	we have that \(\mathbf{M}, v \vDash \phi\)}, 
\end{align*}
where \(R^+\) is the transitive closure of \(R\), i.e.\ \(w R^+ v\) iff there is a (non-empty) sequence \(w_0, \ldots, w_n\) such that \(w = w_0 R \cdots R w_n = v\).

Given a multiset of formulas \(\Gamma\) we will write \(\dnecm \Gamma\) to mean \(\Gamma, \necm\Gamma\).

\begin{definition}
  We define the Hilbert system \(\text{HK}^+\) as the Hilbert system over the language of the master modality with axioms:
  \begin{enumerate}
    \item \(\phi \to (\psi \to \phi)\);
    \item \((\phi \to (\psi \to \chi)) \to ((\phi \to \psi) \to (\phi \to \chi))\);
    \item \(((\phi \to \bot) \to \bot) \to \phi\);
    \item \(\nec(\phi \to \psi) \to (\nec \phi \to \nec \psi)\);
    \item \(\necm(\phi \to \psi) \to (\necm \phi \to \necm \psi)\);
    \item \(\necm \phi \to \nec \phi\);
    \item \(\necm \phi\to \nec\necm \phi \);
    \item \(\nec \phi \to ( \necm (\phi \to \nec \phi) \to \necm \phi)\).
  \end{enumerate}
  And the rules of modus ponens: and necessitation for \(\necm\):
	\[ 
		\AxiomC{\(\phi\)}
		\AxiomC{\(\phi \to \psi\)}
		\RightLabel{MP}
		\BinaryInfC{\(\psi\)}
		\DisplayProof
		\qquad
		\AxiomC{\(\phi\)}
		\RightLabel{Nec}
		\UnaryInfC{\(\necm \phi\)}
		\DisplayProof
	\]	
\end{definition}

As mentioned in~\cite{shamkanov2024realizationtheoremmodallogic}, \(\text{HK}^+\) is sound and weakly complete with respect to the \(\text{K}^+\) semantics.


\subsection*{Shamkanov's non-wellfounded sequent calculus for \(\text{K}^+\)}

In~\cite{shamkanov2024realizationtheoremmodallogic} Shamkanov provided a non-wellfounded Gentzen calculus for \(\text{K}^+\), let us denote it as \(\text{G}^\infty\text{K}^+\) and its version with cut as \(\text{G}^\infty\text{K}^+ + \textsf{Cut}\).
In \(\text{G}^\infty \text{K}^+\) sequents are ordered pairs \((\Gamma, \Delta)\) where \(\Gamma, \Delta\) are finite multisets of formulas.
These are usually denoted as \(\Gamma \Rightarrow \Delta\).
Proofs are non-wellfounded trees, i.e.\ trees that are allowed to have infinite length paths from the root (infinite branches).
It is easy to show that \(\text{G}^\infty \text{K}^+\) is sound and that if \(\text{HK}^+ \vdash \phi\) then \(\text{G}^\infty\text{K}^+ + \text{Cut} \vdash {\Rightarrow \phi}\) (by induction on the size of the Hilbert proof and some inversion principles of \(\text{G}^\infty\text{K}^+ + \text{Cut}\)), so \(\text{G}^\infty\text{K}^+ + \text{Cut}\) is also weakly complete.
Then soundness of \(\text{G}^\infty\text{K}^+\) is clear, (weak) completness will be a consequence of the cut elimination procedure we are going to describe in the paper.\footnote{Note that in~\cite{shamkanov2024realizationtheoremmodallogic} the weak completeness of \(\text{G}^\infty \text{K}^+\) (completeness in the set of theorems but not necessarily in the consequence relation given by the logic) is already proven via refutation trees.
Then one could eliminate cuts via a semantical argument, since (weak) completeness was shown for the cut-free version of the calculus.
}

As usual with non-wellfounded proofs a condition limiting the possible shape of the infinite branches needs to be added, for the sake of soundness.
This condition is simplified by turning to another non-wellfounded Gentzen calculus, which we will denote as \(\text{G}^\infty_\ell \text{K}^+\), whose sequents are \emph{annotated}.
Let us define it precisely.

An \emph{annotation} is just a formula or the symbol \(\circ\), meaning that no formula is annotated.
Sequents in \(\text{G}^\infty_\ell \text{K}^+\) are triples \((\Gamma, s, \Delta)\), where \(\Gamma, \Delta\) are finite multisets of formulas and \(s\) is an annotation.
An additional condition is imposed on sequents, if the annotation of the sequent is a formula \(\phi\) then \(\necm \phi\) must occur in \(\Delta\).
A sequent \((\Gamma, s, \Delta)\) is denoted by \(\Gamma \Rightarrow_s \Delta\).
In the case the annotation is a formula \(\phi\), we say that the formula is in
\emph{focus} and if it is \(\circ\), we say that the sequent is \emph{unfocused}.
Then the rules of \(\text{G}^\infty_\ell \text{K}^+\) are

	\begin{align*}
		\AxiomC{}
		\RightLabel{Ax}
		\UnaryInfC{\(\Gamma, p \Rightarrow_s p, \Delta\)}
		\DisplayProof
		 &  &
		\AxiomC{}
		\RightLabel{Ax-\(\bot\)}
		\UnaryInfC{\(\Gamma, \bot \Rightarrow_s \Delta\)}
		\DisplayProof
		\\
		 &  & \\
		\AxiomC{\(\Gamma \Rightarrow_s \Delta, \phi\)}
		\AxiomC{\(\Gamma, \psi \Rightarrow_s \Delta\)}
		\RightLabel{\(\to\)L}
		\BinaryInfC{\(\Gamma, \phi \to \psi \Rightarrow_s \Delta\)}
		\DisplayProof
		 &  &
		\AxiomC{\(\Gamma, \phi \Rightarrow_s \psi, \Delta\)}
		\RightLabel{\(\to\)R}
		\UnaryInfC{\(\Gamma \Rightarrow_s \phi \to \psi, \Delta\)}
		\DisplayProof
	\end{align*}

	\[
		\AxiomC{\(\Gamma, \dnecm \Pi \Rightarrow_\circ \phi\)}
		\RightLabel{\(\nec\)}
		\UnaryInfC{\(\Sigma, \nec \Gamma, \necm \Pi \Rightarrow_s \nec \phi, \Delta\)}
		\DisplayProof
	\]

	\[
		\AxiomC{\(\Gamma, \dnecm \Pi \Rightarrow_\circ \phi\)}
		\AxiomC{\(\Gamma, \dnecm \Pi \Rightarrow_\phi \necm \phi\)}
		\RightLabel{\(\necm\)}
		\BinaryInfC{\(\Sigma, \nec \Gamma, \necm \Pi \Rightarrow_s \necm \phi, \Delta\)}
		\DisplayProof
	\]

  The condition on infinite branches is then that any branch has a suffix that always has the same formula \(\phi\) in focus.\footnote{
	This particular system is not defined by Shamkanov but is easily infered from
	the systems he defines.
	In particular, in~\cite{shamkanov2024realizationtheoremmodallogic} he defines
	this system with extra annotations since the purpose is to prove a
	realization theorem of justification logic.
	In~\cite{shamkanov2023structuralprooftheorymodal} he defines a system of
	\(\text{K}^+\) with the same annotations but the sequents have an extra
	(possibly infinite) set of formulas in order to establish a connection with an
	\(\omega\)-rule system for \(\text{K}^+\).
  In any case, the definition of the system here is to have the exact and
	precise definition for our purpose, but it is not original in any sense.}

We are allowed to move from studying cut elimination in \(\text{G}^\infty \text{K}^+\) to study cut elimination in \(\text{G}^\infty_\ell \text{K}^+\) since it can be shown that
\[ 
  \text{G}^\infty \text{K}^+ \vdash \Gamma \Rightarrow \Delta \text{ iff } \text{G}^\infty_\ell \text{K}^+ \vdash \Gamma \Rightarrow_\circ \Delta.
\]
and the same if we have the cut rule in both.
For the right to left we just need to delete the annotations and notice that the branch condition of \(\text{G}^\infty_\ell \text{K}^+\) implies the branch condition of \(\text{G}^\infty \text{K}^+\).
For the left to right notice that by the shape of the rules there is only one possible way of annotating a proof in \(\text{G}^\infty \text{K}^+\) such that the root is annotated with \(\circ\).
It is straightforward to see that the branch condition of \(\text{G}^\infty \text{K}^+\) will imply that the branch condition of \(\text{G}^\infty_\ell \text{K}^+\) is fulfilled in the annotated version of the proof.

\subsection*{Cut elimination for local-progress systems}

As we mentioned in the introduction, we want to use our previously defined tools of~\cite{previous} to prove cut elimination.
Let us briefly introduce these tools here.

We work with local-progress non-wellfounded sequent calculi, let us fix an arbitrary one and call it \(\mathcal{C}\).
As usual, \(\mathcal{C}\) has a set \(\mathcal{R}\) whose elements are sequent rules, which we call the \emph{rules of\/ \(\mathcal{C}\)}.
A pre-proof in \(\mathcal{C}\) is a (possibly non-wellfounded) tree whose nodes are labelled by an ordered pair consisting of a sequent and a rule from \(\mathcal{R}\) such that for any node with sequent \(S\) and rule \(R\) if \(S_0, \ldots, S_{n-1}\) are the sequents of its successors (in order) then \((S_0, \ldots, S_{n-1}, S)\) must be a rule instance of \(R\).

In addition, \(\mathcal{C}\) has a function \(L\) called the \emph{local progress function}.
Given a rule instance \(r\) with premises \(S_0, \ldots, S_{n-1}\) and conclusion \(S\) of some rule \(R \in \mathcal{R}\), we have that \(L_R(S_0, \ldots, S_{n-1}, S) \subseteq \{0, \ldots, n-1\}\).\footnote{We are denoting \(L(R,S_0,\ldots,S_{n-1},S)\) as \(L_R(S_0,\ldots,S_{n-1},S)\). In other words, progress does not only depend on the rule instance but also on the exact rule applied, as in principle to different rules can share some rule instances.}
We interpret this function as follows.
Imagine we have a pre-proof in
\(\mathcal{C}\) and \(w\) is a node with \(n\) successors.
The node \(w\) is labelled with the sequent \(S\) and rule \(R\) and each of the
successors is labelled with the sequent \(S_0, \ldots, S_{n-1}\), respectively.
Then we will say that from \(w\) to its \(i\)-th successor there is
\emph{progress} iff \(i \in L_R(S_0, \ldots, S_{n-1}, S)\).
A \emph{proof} in \(\mathcal{C}\) is a pre-proof such that in any infinite
branch there is infinitely often progress from a node in the branch to its
successor in the branch.

When defining local progress system it will be usual to describe the rules as
\[ 
  \AxiomC{\(\mathcal{S}_0\)}
  \AxiomC{\(\cdots\)}
  \AxiomC{\(\mathcal{S}_{n-1}\)}
  \RightLabel{R}
  \TrinaryInfC{\(\mathcal{S}\)}
  \DisplayProof
\]
where \(S_0, \ldots, S_{n-1}, S\) are sequent-schemes.
This defines the rule \(R\) whose rules instances are the tuples \((S_0, \ldots, S_{n-1}, S)\) where \(S_i\) is an instantiation of \(\mathcal{S}_i\) and \(S\) is an instantiation of \(\mathcal{S}\).
In addition we will say that progress only occurs at the \(i_0, \ldots, i_k\)-th premises of \(R\) to mean that \(L_R(S_0, \ldots, S_{n-1}, S) = \{i_0, \ldots, i_k\}\), where \((S_0, \ldots, S_{n-1}, S)\) is any rule instance of \(R\).
In case we do not mention anything about progress in a rule \(R\) or we say that in \(R\) there is no progress, we mean that \(L_R(r) = \varnothing\) for any rule instance \(r\) of \(R\).
\footnote{In this paper we will work with a system where progress will occur at most in one premise of the rule instance.}

Any proof in a local-progress system can be divided into (possibly infinitely
many) finite trees, creating a partition of the nodes of the proof.
An element of this partition is called a \emph{local fragment} of the
proof, the idea is that two nodes are in the same local fragment iff there is a
(non-directed) path in the tree that allows to go from one node to the other
without encountering progress.
In particular, note that inside a local fragment no progress can occur and that
all progress in the proof occurs from a leaf of a local fragment to the root
of another local fragment.
The \emph{main local fragment} of a proof is the fragment starting at the root.
The \emph{local height} of the proof will be the height of the main local
fragment, for a proof \(\pi\) this height will be denoted as \(|\pi|\).

The tools that we developed in~\cite{previous} basically allowed us to do
corecursion at the level of local fragments.
In other words, in a corecursive step we will give not only the last node of
the resulting proof after the corecursion, but the whole main local fragment.
For details consult Section 3 in~\cite{previous}.

\section{Definition of \(\alpha\textsf{-}\text{G}^{ \infty }_{ \ell }\text{K}^{ + }_{ s }\)}\label{sec:definition}

	For the Gentzen calculi we are going to define we need to consider the concept of proof with witnesses.
	Usually a Gentzen calculi rule is of the shape
	\[ 
		\AxiomC{\(S_0\)}
		\AxiomC{\(\cdots\)}
		\AxiomC{\(S_{n-1}\)}
		\RightLabel{R}
		\TrinaryInfC{\(S\)}
		\DisplayProof
	\]
	However, we are going to use rules of shape
	\[ 
		\AxiomC{}
		\UnaryInfC{\(\pi_0 \vdash S_{0}\)}
		\AxiomC{\(\cdots\)}
		\AxiomC{}
		\UnaryInfC{\(\pi_{m-1} \vdash S_{m-1}\)}
		\AxiomC{\(S_m\quad \cdots\quad S_{m + n -1}\)}
		\RightLabel{R}
		\QuaternaryInfC{\(S\)}
		\DisplayProof
	\]
  where \(\pi_0, \ldots, \pi_{m-1}\) will be proofs of \(S_0, \ldots, S_{m-1}\), respectively, in other Gentzen calculi(which will be defined beforehand, generating a hierarchy of sequent calculi defined recursively).
	In the proof tree this will look as a node with the label
	\((S, R, \pi_0, \ldots, \pi_{m-1})\) and \(n\) successors, labelled with
	sequents \(S_m\) to \(S_{m+n-1}\) respectively.
	
	In such a rule instance \(\pi_0, \ldots, \pi_{m-1}\) are called
	\emph{witnesses}.
	Given a proof in a system with this kind of rules we will sometimes call it a
	\emph{proof with witnesses} and the witnesses are all the witnesses that occur
	at some of its rule instances.
	The \emph{main global fragment} of such a proof is its structure without
	considering the witnesses.

	When we talk about a subproof in a calculus of this kind we mean a proof
	generated by taking one of its nodes as the root.
	In particular, these nodes always belong to the main global fragment, so
	the witnesses are not subproofs.

	Note that in a proof \(\pi\) we may have that a witness \(\tau\) is a proof with witnesses so it also has witnesses.
	The witnesses of \(\tau\) are not witnesses of \(\pi\) (i.e.\ being a witness is not transitive).
	In this setting we have to understand that a cut-free proof is not only a proof without instances of cut, it is a proof without instances of cut such that all its witnesses have no instances of cut either, neither the witnesses of witnesses,\ldots and so on.

\begin{definition}\label{def:ordinal-system}
	For each ordinal \(\alpha\) and annotation  \(s\),  we define the system \(\alpha\text{-G}^{ \infty }_{ \ell }\text{K}^{ + }_{ s }\) as the local progress system with rules:
	\begin{align*}
		\AxiomC{}
		\RightLabel{Ax}
		\UnaryInfC{\(\Gamma, p \Rightarrow_s p, \Delta\)}
		\DisplayProof
		 &  &
		\AxiomC{}
		\RightLabel{Ax-\(\bot\)}
		\UnaryInfC{\(\Gamma, \bot \Rightarrow_s \Delta\)}
		\DisplayProof
		\\
		 &  & \\
		\AxiomC{\(\Gamma \Rightarrow_s \Delta, \phi\)}
		\AxiomC{\(\Gamma, \psi \Rightarrow_s \Delta\)}
		\RightLabel{\(\to\)L}
		\BinaryInfC{\(\Gamma, \phi \to \psi \Rightarrow_s \Delta\)}
		\DisplayProof
		 &  &
		\AxiomC{\(\Gamma, \phi \Rightarrow_s \psi, \Delta\)}
		\RightLabel{\(\to\)R}
		\UnaryInfC{\(\Gamma \Rightarrow_s \phi \to \psi, \Delta\)}
		\DisplayProof
	\end{align*}

	\[
		\AxiomC{}
		\UnaryInfC{\(\tau \vdash^{ \beta } \Gamma, \dnecm \Pi \Rightarrow_\circ \phi\)}
		\RightLabel{\(\nec\)}
		\UnaryInfC{\(\Sigma, \nec \Gamma, \necm \Pi \Rightarrow_s \nec \phi, \Delta\)}
		\DisplayProof
	\]

	\[
		\AxiomC{}
		\UnaryInfC{\(\tau \vdash^{ \beta } \Gamma, \dnecm \Pi \Rightarrow_\circ \phi\)}
		\AxiomC{\(\Gamma, \dnecm \Pi \Rightarrow_\phi \necm \phi\)}
		\LeftLabel{\(s = \phi\)}
		\RightLabel{\(\necm_f\)}
		\BinaryInfC{\(\Sigma, \nec \Gamma, \necm \Pi \Rightarrow_s \necm \phi, \Delta\)}
		\DisplayProof
	\]

	\[
		\AxiomC{}
		\UnaryInfC{\(\tau_0 \vdash^{ \beta } \Gamma, \dnecm \Pi \Rightarrow_\circ \phi\)}
		\AxiomC{}
		\UnaryInfC{\(\tau_1 \vdash^{ \gamma } \Gamma, \dnecm \Pi \Rightarrow_\phi \necm \phi\)}
		\LeftLabel{\(s \neq \phi\)}
		\RightLabel{\(\necm_u\)}
		\BinaryInfC{\(\Sigma, \nec \Gamma, \necm \Pi \Rightarrow_s \necm \phi, \Delta\)}
		\DisplayProof
	\]

	where \(\beta,\gamma < \alpha\) and \(\tau \vdash^{ \alpha }\Gamma \Rightarrow_{ s' }\Delta\) means that \(\tau\) is a proof of \(\Gamma \Rightarrow_{ s' }\Delta\) in \(\alpha\text{-G}^{ \infty }_{ \ell }\text{K}^{ + }_{ s' }\).
	We note that \(\Gamma, \Delta, \Sigma, \Pi\) are finite multisets of formulas and, in particular, may be empty.

	In instances of the rules \(\to\text{L}\), \(\to\text{R}\), \(\nec\),
	\(\necm_f\) and \(\necm_u\), we define the \emph{principal formula} to be the
	displayed formula in the conclusion.
%
	In the modal rules, the multisets \(\Sigma\) and \(\Delta\) will be called
	the \emph{weakening part} of the rule instance.

	Progress occurs at the right premise of the \(\necm_f\) rule. Notice that \(\necm_f\) and \(\necm_u\) have side conditions \(s = \phi\) and \(s \neq \phi\), respectively. This, together with the progress condition implies that if \(s = \circ\) the system is in fact finitary (every branch will be finite).
\end{definition}

As we stated in the preliminaries, a proof in a local-progress calculus can be
split into local fragments.
In particular, each proof in
\(\alpha\text{-G}^\infty_\ell \text{K}^+_s\), has a local fragment starting at
the root and local fragments starting at each right premise of a \(\necm_f\)
rule instance.

The \(s\) in \(\alpha\text{-}\text{G}^\infty_\ell\text{K}^+_s\) determines which
annotation is in focus, so when the annotation in focus is changed we are forced
to go to a witness with a different annotation (see the \(\necm_u\)~rule).

\begin{definition}[Adding cuts]
	We define the cut rule
	\[
		\AxiomC{\(\Gamma \Rightarrow_s \Delta, \chi\)}
		\AxiomC{\(\chi, \Gamma \Rightarrow_s \Delta\)}
		\RightLabel{\(\textsf{Cut}\)}
		\BinaryInfC{\(\Gamma \Rightarrow_s \Delta\)}
		\DisplayProof
	\]

	which makes no progress.

	For each \(\alpha\) we define the systems:
	\begin{enumerate}
		\item We define \(\alpha\text{-G}^\infty_\ell \text{K}^+_s + \textsf{Cut}\) to be the systems as in Definition~\ref{def:ordinal-system} with adding \(\textsf{Cut}\) to the list of rules (in particular, \textsf{Cut} may be used in the main global fragment, in witnesses, in witnesses of witnesses, and so on).
		\item \(\alpha\text{-G}^\infty_\ell \text{K}^+_s + \textsf{wCut}\) are the system as in Definition~\ref{def:ordinal-system} with allowing the witnesses to belong to \(\beta\text{-G}^\infty_\ell \text{K}^+_s + \textsf{Cut}\) for \(\beta < \alpha\) and \textbf{not} adding the \(\textsf{Cut}\) rule.
		\item \(\alpha\text{-G}^\infty_\ell \text{K}^+_s + \textsf{mCut}\) are the systems as in Definition~\ref{def:ordinal-system} with the same witnesses \textbf{but} adding the \(\textsf{Cut}\) rule, i.e.~\(\textsf{Cut}\) may occur in the main global fragment but the witnesses must be cut-free.
	\end{enumerate}

	A proof in the second system is said to have \emph{witnesses-cuts} only and a proof in the third system is said to have \emph{main-cuts} only.
\end{definition}

\begin{definition}
	We say that \(\pi\) is a proof in
	\(\text{G}^{ \infty }_{ \ell }\text{K}^{ + }_{ s }\) iff there is an
	\(\alpha\) such that \(\pi\) is a proof in
	\(\alpha\text{-G}^{ \infty }_{ \ell }\text{K}^{ + }_{ s }\) and similarly iff
	we add \textsf{Cut}, \textsf{mCut} or \textsf{wCut}.
	If \(\pi\) is a proof in \(\text{G}^{ \infty }_{ \ell }\text{K}^{ + }_{ s } + \textsf{Cut}\) we will denote by \(\|\pi\|\) the minimum ordinal \(\alpha\) such that \(\pi\) is a proof in \(\alpha\text{-G}^{ \infty }_{ \ell }\text{K}^{ + }_{ s } + \textsf{Cut}\).
	This ordinal is called the \emph{(ordinal) height} of the proof.
	The (ordinal) height of a node in a proof is the height of the subproof at
	that node.
\end{definition}

Given a cut in a proof we can talk about its \emph{size} (which is the size of its cut formula) and its \emph{height} (which is the Hessenberg sum of the ordinal height of its premises).

The next proposition claims trivial facts about the three possible ways of
adding cut to systems \(\alpha\text{-}\text{G}^\infty_\ell\text{K}^+_s\)
and establishes the connection to system \(\text{G}^\infty_\ell \text{K}^+\).

\begin{proposition}
	We have that
	\begin{enumerate}
		\item If \(\pi\) is a proof in \(\alpha\text{-G}^\infty_\ell\text{K}^+_s\) then it is also a proof in \(\alpha\text{-G}^\infty_\ell\text{K}^+_s + \textsf{wCut}\) and in \(\alpha\text{-G}^\infty_\ell\text{K}^+_s + \textsf{mCut}\).
		\item If \(\pi\) is a proof in \(\alpha\text{-G}^\infty_\ell\text{K}^+_s + \textsf{wCut}\) or in \(\alpha\text{-G}^\infty_\ell\text{K}^+_s + \textsf{mCut}\) then it is a proof in \(\alpha\text{-G}^\infty_\ell\text{K}^+_s + \textsf{Cut}\).
		\item \(\text{G}^\infty_\ell \text{K}^+_s \vdash \Gamma \Rightarrow_s \Delta\) iff \(\text{G}^\infty_\ell \text{K}^+ \vdash \Gamma \Rightarrow_s \Delta\) and similarly if we add \(\textsf{Cut}\).
	\end{enumerate}
\end{proposition}
\begin{proof}
  The first two claims are straightforward by definition.
	The third claim is proven by induction on \(\|\pi\|\) where for \(\pi\) in \(\text{G}^\infty_\ell \text{K}^+\) (\(+ \textsf{Cut}\)), \(\|\pi\|\) refers to the 
	the definition of height in~\cite{shamkanov2023structuralprooftheorymodal}.
\end{proof}

Our objective will be to show that if \(\text{G}^\infty_\ell \text{K}^+_s  + \textsf{Cut} \vdash \Gamma \Rightarrow_s \Delta\) then \(\text{G}^\infty_\ell \text{K}^+_s \vdash \Gamma \Rightarrow_s \Delta\).
Broadly, we will prove this in three steps (see Theorem~\ref{th:cut-elimination} and Figure~\ref{fig:cut-elim}):
\begin{enumerate}
	\item If \(\text{G}^\infty_\ell \text{K}^+_s  + \textsf{Cut} \vdash \Gamma \Rightarrow_s \Delta\)	then  \(\text{G}^\infty_\ell \text{K}^+_s  + \textsf{mCut} \vdash \Gamma \Rightarrow_s \Delta\);
	\item If \(\text{G}^\infty_\ell \text{K}^+_s  + \textsf{mCut} \vdash \Gamma \Rightarrow_s \Delta\)	then  \(\text{G}^\infty_\ell \text{K}^+_s  + \textsf{wCut} \vdash \Gamma \Rightarrow_s \Delta\), with the additional property that the proof in \(\text{G}^\infty_\ell \text{K}^+_s  + \textsf{wCut}\) will only have finitely many cuts in each witness; and 
	\item If \(\text{G}^\infty_\ell \text{K}^+_s  + \textsf{wCut} \vdash \Gamma \Rightarrow_s \Delta\) with only finitely many cuts in each witness, then (using
	cut admissibility) 
	\(\text{G}^\infty_\ell \text{K}^+_s  \vdash \Gamma \Rightarrow_s \Delta\).
\end{enumerate}

\textbf{On the use of the expressions ``cut admissibility'' and ``cut-elimination''}.
In the proof theoretical literature it is usual to distinguish ``cut admissibility'' and ``cut elimination''.
It is common to understand cut-elimination as an algorithmic procedure operating on proofs to eliminate the cut rule from a sequent calculus and cut admissibility simply as proving (not necessarily in an algorithmic way) that the cut rule is not necessary, e.g. via the semantics.
This is not the meaning that we will give to this expressions. To be more precise, if \(\mathcal{G}\) is a sequent calculus, cut admissibility for \(\mathcal{G}\) is the statement that if \(\mathcal{G} \vdash \Gamma \Rightarrow \Delta, \chi\) and \(\mathcal{G} \vdash \chi, \Gamma \Rightarrow \Delta\) then \(\mathcal{G} \vdash \Gamma \Rightarrow \Delta\).
On the other hand, cut elimination is the statement that the sequents provable in \(\mathcal{G} + \textsf{Cut}\) (i.e., the system \(\mathcal{G}\) extended by the cut rule) are also provable in \(\mathcal{G}\).

In wellfounded proof theory both statements are can be shown to be equivalent by using an induction on the height of proofs.
However, in the non-wellfounded setting, both statements are not equivalent and elimination is strictly stronger than admissibility.
Our method can be seen as a way of lifting admissibility into elimination.
Also, we would like to notice that if a non-wellfounded proof has finitely many cuts, admissibility suffices to produce a cut-free proof.

\section{Change of annotations}\label{sec:change-of-annotations}

Sometimes we have a proof of a sequent \(\Gamma \Rightarrow_s \Delta\) and we need a proof of \(\Gamma \Rightarrow_{s'} \Delta\), i.e.\ we need to change the annotation of the sequent.
As we will show in the following definition, it is possible to define such a translation of proofs.
However, it may change the (ordinal) height of the proof.

\begin{definition}
Given \(\pi\) in \(\text{G}^{ \infty }_{ \ell }\text{K}^{ + }_{ s }\) we can define the proof \(\pi^{ s' }\) of the same sequent in \(\text{G}^{ \infty }_{ \ell }\text{K}^{ + }_{ s' }\).
If \(s = s'\) we just return the same proof, otherwise we proceed by induction on the local height and cases in the last rule applied:
\[
\AxiomC{}
\RightLabel{Ax}
\UnaryInfC{\(\Gamma, p \Rightarrow_s p, \Delta\)}
\DisplayProof
\mapsto
\AxiomC{}
\RightLabel{Ax}
\UnaryInfC{\(\Gamma, p \Rightarrow_{s'} p, \Delta\)}
\DisplayProof
\]

\[
\AxiomC{}
\RightLabel{Ax-\(\bot\)}
\UnaryInfC{\(\Gamma, \bot \Rightarrow_s \Delta\)}
\DisplayProof
\mapsto
\AxiomC{}
\RightLabel{Ax-\(\bot\)}
\UnaryInfC{\(\Gamma, \bot \Rightarrow_{s'} \Delta\)}
\DisplayProof
\]

\[
\AxiomC{\(\pi_{ 0 }\)}
\noLine
\UnaryInfC{\(\Gamma \Rightarrow_s \Delta, \phi\)}
\AxiomC{\(\pi_{ 1 }\)}
\noLine
\UnaryInfC{\(\Gamma, \psi \Rightarrow_s \Delta\)}
\RightLabel{\(\to\)L}
\BinaryInfC{\(\Gamma, \phi \to \psi \Rightarrow_s \Delta\)}
\DisplayProof
\mapsto
\AxiomC{\(\pi_{ 0}^{ s' }\)}
\noLine
\UnaryInfC{\(\Gamma \Rightarrow_{ s'} \Delta, \phi\)}
\AxiomC{\(\pi_{ 1 }^{ s' }\)}
\noLine
\UnaryInfC{\(\Gamma, \psi \Rightarrow_{s'} \Delta\)}
\RightLabel{\(\to\)L}
\BinaryInfC{\(\Gamma, \phi \to \psi \Rightarrow_{s'} \Delta\)}
\DisplayProof
\]

\[
\AxiomC{\(\pi_{ 0 }\)}
\noLine
\UnaryInfC{\(\Gamma, \phi \Rightarrow_s \psi, \Delta\)}
\RightLabel{\(\to\)R}
\UnaryInfC{\(\Gamma \Rightarrow_s \phi \to \psi, \Delta\)}
\DisplayProof
\mapsto
\AxiomC{\(\pi_{ 0 }^{ s' }\)}
\noLine
\UnaryInfC{\(\Gamma, \phi \Rightarrow_{s'} \psi, \Delta\)}
\RightLabel{\(\to\)R}
\UnaryInfC{\(\Gamma \Rightarrow_{s'} \phi \to \psi, \Delta\)}
\DisplayProof
\]

\[
\AxiomC{}
\UnaryInfC{\(\tau \vdash \Gamma, \dnecm \Pi \Rightarrow_\circ \phi\)}
\RightLabel{\(\nec\)}
\UnaryInfC{\(\Sigma, \nec \Gamma, \necm \Pi \Rightarrow_s \nec \phi, \Delta\)}
\DisplayProof
\mapsto 
\AxiomC{}
\UnaryInfC{\(\tau \vdash \Gamma, \dnecm \Pi \Rightarrow_\circ \phi\)}
\RightLabel{\(\nec\)}
\UnaryInfC{\(\Sigma, \nec \Gamma, \necm \Pi \Rightarrow_{s'} \nec \phi, \Delta\)}
\DisplayProof
\]

If \(\pi\) is:
\[
\AxiomC{}
\UnaryInfC{\(\tau \vdash \Gamma, \dnecm \Pi \Rightarrow_\circ \phi\)}
\AxiomC{\(\pi_{ 0 }\)}
\noLine
\UnaryInfC{\(\Gamma, \dnecm \Pi \Rightarrow_\phi \necm \phi\)}
\RightLabel{\(\necm_f\)}
\BinaryInfC{\(\Sigma, \nec \Gamma, \necm \Pi \Rightarrow_s \necm \phi, \Delta\)}
\DisplayProof
\]
(so \(s = \phi\) and then \(s' \neq \phi\)) then it maps to
\[
\AxiomC{}
\UnaryInfC{\(\tau \vdash \Gamma, \dnecm \Pi \Rightarrow_\circ \phi\)}
\AxiomC{}
\UnaryInfC{\(\pi_{ 0 }\vdash \Gamma, \dnecm \Pi \Rightarrow_\phi \necm \phi\)}
\RightLabel{\(\necm_u\)}
\BinaryInfC{\(\Sigma, \nec \Gamma, \necm \Pi \Rightarrow_{s'} \necm \phi, \Delta\)}
\DisplayProof
\]

If \(\pi\) is:
\[
\AxiomC{}
\UnaryInfC{\(\tau_0 \vdash \Gamma, \dnecm \Pi \Rightarrow_\circ \phi\)}
\AxiomC{}
\UnaryInfC{\(\tau_1 \vdash \Gamma, \dnecm \Pi \Rightarrow_\phi \necm \phi\)}
\RightLabel{\(\necm_u\)}
\BinaryInfC{\(\Sigma, \nec \Gamma, \necm \Pi \Rightarrow_s \necm \phi, \Delta\)}
\DisplayProof
\]
then it maps to
\[
\AxiomC{}
\UnaryInfC{\(\tau_0 \vdash \Gamma, \dnecm \Pi \Rightarrow_\circ \phi\)}
\AxiomC{}
\UnaryInfC{\(\tau_1 \vdash \Gamma, \dnecm \Pi \Rightarrow_\phi \necm \phi\)}
\RightLabel{\(\necm_u\)}
\BinaryInfC{\(\Sigma, \nec \Gamma, \necm \Pi \Rightarrow_{s'} \necm \phi, \Delta\)}
\DisplayProof
\]
in case \(s' \neq \phi\), or it maps to
\[
\AxiomC{}
\UnaryInfC{\(\tau_0 \vdash \Gamma, \dnecm \Pi \Rightarrow_\circ \phi\)}
\AxiomC{\(\tau_1\)}
\noLine
\UnaryInfC{\(\Gamma, \dnecm \Pi \Rightarrow_\phi \necm \phi\)}
\RightLabel{\(\necm_f\)}
\BinaryInfC{\(\Sigma, \nec \Gamma, \necm \Pi \Rightarrow_{s'} \necm \phi, \Delta\)}
\DisplayProof
\]
in case \(s' = \phi\).

\[
\AxiomC{\(\pi_{ 0 }\)}
\noLine
\UnaryInfC{\(\Gamma \Rightarrow_{ s } \Delta, \chi\)}
\AxiomC{\(\pi_{ 1 }\)}
\noLine
\UnaryInfC{\(\chi, \Gamma \Rightarrow_{ s }\Delta\)}
\RightLabel{Cut}
\BinaryInfC{\(\Gamma \Rightarrow_{ s }\Delta\)}
\DisplayProof
\mapsto
\AxiomC{\(\pi_{ 0 }^{ s' }\)}
\noLine
\UnaryInfC{\(\Gamma \Rightarrow_{ s' } \Delta, \chi\)}
\AxiomC{\(\pi_{ 1}^{ s' }\)}
\noLine
\UnaryInfC{\(\chi, \Gamma \Rightarrow_{ s' }\Delta\)}
\RightLabel{Cut}
\BinaryInfC{\(\Gamma \Rightarrow_{ s' }\Delta\)}
\DisplayProof
\]
\end{definition}

After a change of annotation the height of the proof may be altered.
In the next lemma we prove some bounds to its change.

\begin{lemma}
Let \(\phi\) be a formula and \(s\) be an annotation.
We have that
\begin{enumerate}
\item If \(\pi\) is a proof in \(\text{G}^{ \infty }_{ \ell }\text{K}^{ + }_{ \phi }\), then \(\|\pi^{ s }\| \leq \|\pi\| + 1\).
\item If \(\pi\) is a proof in \(\text{G}^{ \infty }_{ \ell }\text{K}^{ + }_{ \circ }\), then \(\|\pi^{ s }\| \leq \|\pi\|\).
\end{enumerate}
\end{lemma}
\begin{proof}
By induction on the local height.
In 1.\ the possible increase occurs in the \(\necm_{f}\)-case since it can turn a subproof \(\pi'\) with (possibly) \(\|\pi'\| = \|\pi\|\) to a witness which will impose that the new height is strictly bigger ordinal.
Notice that in 2.\ this is not possible since there are no instances of \(\necm_f\) in the original proof, only of \(\necm_{u}\).
\end{proof} 

In the future, we will need to apply this construction while we prove cut admissibility and cut elimination.
However, in order to use this translation adequately in the proofs, we will need to have some additional properties.
For example, if we apply this translation to a proof without cuts it must remain without cuts after the translation.
We will now describe the conditions that we need for this translation and other auxiliary translations such as weakening, contraction and inversion.
But first we need to define a special type of cut, local cuts:

\begin{definition}
  Let \(\pi\) be a proof in \(\text{G}^\infty_\ell \text{K}^+_s + \text{Cut}\).
  We say that \(\pi\) has \emph{local cuts only} in case \(\pi\) is a proof in \(\text{G}^\infty_\ell \text{K}^+_s + \text{mCut}\) and all the instances of cut
  occurs at the main local fragment (i.e.\ before any modal rule).
  The cuts occuring in the main local fragment of a proof will be called the \emph{local cuts} of the proof.
\end{definition}

\begin{definition}
  Let \(f\) be an \(n\)-ary function whose domain is some set of tuples of proofs and its range is some sets of proofs.
  We say that \(f\) preserves
  \begin{enumerate}
    \item \emph{Ordinal height} iff given a tuple \((\pi_0, \ldots, \pi_{n-1})\)
    in the domain of \(f\) the height of \(f(\pi_0, \ldots, \pi_{n-1})\) is
    smaller or equal than the maximum of heights of the \(\pi_i\)s.
    \item \emph{Local height} iff given a tuple \((\pi_0, \ldots, \pi_{n-1})\)
      in the domain of \(f\) the local height of \(f(\pi_0, \ldots, \pi_{n-1})\) is
    smaller or equal than the maximum of local heights of the \(\pi_i\)s.
    \item \emph{Size of cuts} iff given a tuple \((\pi_0, \ldots, \pi_{n-1})\) in the domain of \(f\) such that each \(\pi_i\) has all its cuts are of size smaller than \( n\), then \(f(\pi_0, \ldots, \pi_{n-1})\) has also all its cuts of size smaller than \(n\).\footnote{With all its cuts here we mean the cuts in the main global fragment of the proof, in the witnesses, in the witnesses of witnesses, and so on.}
    In particular, if the \(\pi_i\)s are cut-free (i.e.\ the sizes of cuts are \(< 0\)) then \(f(\pi_0, \ldots, \pi_{n-1})\) is also cut-free.
    \item \emph{Freeness of cuts in the main local fragment} iff given a tuple \((\pi_0, \ldots, \pi_{n-1})\) in the domain of \(f\) such that each \(\pi_i\) without cuts in its main local fragment, then \(f(\pi_0, \ldots, \pi_{n-1})\) does not have cuts in its main local fragment either.
    \item \emph{Locality of cuts} iff given a tuple \((\pi_0, \ldots, \pi_{n-1})\) in the domain of \(f\) such that each \(\pi_i\) has local cuts only, then \(f(\pi_0, \ldots, \pi_{n-1})\) also has local cuts only.
    \item \emph{Locality of cuts in witnesses} iff given a tuple \((\pi_0, \ldots, \pi_{n-1})\) in the domain of \(f\) such that each \(\pi_i\) has witnesses with local cuts only, then the witnesses of \(f(\pi_0, \ldots, \pi_{n-1})\) have local cuts only.         
    \end{enumerate}
    If \(f\) preserves local height, sizes of cuts, freeness of cuts in the
    main local fragment and locality of cuts we say that it is \emph{weakly 
    preserving}.
    If \(f\) preserves all the properties described above we will say that it is
    \emph{strongly preserving}.
\end{definition}

The following observation is straightforward from the definition.

\begin{proposition}
Change of annotations, i.e.\ the function \(\pi \mapsto \pi^s\), is weakly preserving.
\end{proposition}
\begin{proof}
  That it preserves sizes of cuts is trivial from the definition.
  Preservation of local height is due to the transformation converts one rule application in one rule application and modal rules into modal rules, this last property makes the main local fragment to be at the same position (as the main local fragment is determined by where modal rules occur).
  Freeness of cuts in the main local fragment and locality of cuts are preserved since the transformation introduces cuts just from cuts that where already there and, as stated before, it keeps the main local fragment at the same position.
\end{proof}

\section{Auxiliary functions}

In this section we just state one big lemma with all the auxiliary functions that we will need (apart from change of annotations, see Section~\ref{sec:change-of-annotations}).
All the auxiliary functions can be defined straightforwardly by recursion on the local height and all its properties can be shown straightforwardly by induction on the local height.
If the reader wants to see explicit definitions they are encouraged to look in
the appendix chapters A, B and C.

\begin{lemma}
Let \(p\) be a propositional variable,  \(\chi, \chi_0, \chi_1\) be formulas and \(\Gamma, \Gamma', \Delta, \Delta'\) be multisets of formulas.
There are \textbf{strongly preserving} functions \(\textsf{wk}_{\Gamma';\Delta'}\), \(\textsf{lctr}_p\), \(\textsf{rctr}_p\), \(\textsf{rctr}_{\necm \chi}\), \(\textsf{inv}_\bot\), \(\textsf{linv}^0_{\chi_0 \to \chi_1}\), \(\textsf{linv}^1_{\chi_0 \to \chi_1}\), \(\textsf{rinv}_{\chi_0 \to \chi_1}\) from proofs in \(\text{G}^\infty_\ell \text{K}^+_s + \textsf{Cut}\) to proofs in \(\text{G}^\infty_\ell \text{K}^+_s + \textsf{Cut}\) such that:
\begin{enumerate}
  \item \(\pi \vdash \Gamma \Rightarrow_s \Delta\) implies \(\textsf{wk}_{\Gamma';\Delta'}(\pi) \vdash \Gamma, \Gamma' \Rightarrow_s \Delta, \Delta'\).
  \item \(\pi \vdash \Gamma, p, p \Rightarrow_s \Delta\) implies \(\textsf{lctr}_{p}(\pi) \vdash \Gamma, p \Rightarrow_s \Delta\).
  \item \(\pi \vdash \Gamma \Rightarrow_s \Delta, p, p\) implies \(\textsf{rctr}_{p}(\pi) \vdash \Gamma \Rightarrow_s \Delta, p\).
  \item \(\pi \vdash \Gamma \Rightarrow_s \Delta, \necm \chi, \necm \chi\) implies \(\textsf{rctr}_{\necm \chi}(\pi) \vdash \Gamma \Rightarrow_s \Delta, \necm \chi\).
  \item \(\pi \vdash \Gamma \Rightarrow_s \Delta, \bot\) implies \(\textsf{ctr}_{\bot}(\pi) \vdash \Gamma \Rightarrow_s \Delta\).
  \item \(\pi \vdash \Gamma, \chi_0 \to \chi_1 \Rightarrow_s \Delta\) implies \(\textsf{linv}^0_{\chi_0 \to \chi_1}(\pi) \vdash \Gamma \Rightarrow_s \Delta, \chi_0\).
  \item \(\pi \vdash \Gamma, \chi_0 \to \chi_1 \Rightarrow_s \Delta\) implies \(\textsf{linv}^1_{\chi_0 \to \chi_1}(\pi) \vdash \Gamma, \chi_1 \Rightarrow_s \Delta\).
  \item \(\pi \vdash \Gamma \Rightarrow_s \Delta, \chi_0 \to \chi_1\) implies \(\textsf{rinv}_{\chi_0 \to \chi_1}(\pi) \vdash \Gamma, \chi_0 \Rightarrow_s \Delta, \chi_1\).
\end{enumerate}
\end{lemma} 

\section{Pushing cuts}

The first task to prove cut elimination will be to show that it is possible to
push cuts to upper parts of the proof.
We will do this in two stages: first pushing them outside the main local fragment
recursively and then pushing them outside the main global fragment to witnesses
corecursively.

\subsection{Pushing a cuts outside main local fragment}

\begin{definition}
	Let \(\pi \vdash \Gamma \Rightarrow_s \Delta, \chi\) and \(\tau \vdash \chi, \Gamma \Rightarrow_s \Delta\).
	We define \(\textsf{cut}^\chi(\pi, \tau)\) to be the result of applying the rule \(\textsf{Cut}\) to \(\pi\) and \(\tau\).

	Notice that it preserves locality of cuts.
\end{definition}

We start by pushing only one cut that we assume to be a top-most cut in the
main local fragment, i.e.\ there are no cuts above it belonging to the main local
fragment.

\begin{lemma}
	There is a binary function \(\textsf{push-top}_\chi\) whose domain are the pairs \((\pi, \tau)\) of proofs such that
	\begin{enumerate}
		\item \(\pi\) and \(\tau\) are proofs in \(\text{G}^\infty_\ell \text{K}^+_s + \textsf{Cut}\) with no cuts in their main local fragment, and
		\item we have that \(\pi \vdash \Gamma \Rightarrow_s \Delta, \chi\) and \(\tau \vdash \chi, \Gamma \Rightarrow_s \Delta\) for some \(\Gamma, \Delta\), 
		\end{enumerate}  
	and that returns a proof in \(\textsf{G}^\infty_\ell\text{K}^+_s + \textsf{Cut}\) with the following properties: 
	\begin{enumerate}
		\item \(\textsf{push-top}_\chi(\pi, \tau) \vdash \Gamma \Rightarrow_s \Delta\) and 
		has no cuts in its main local fragment,
		\item if the size of cuts in \(\pi\), the size of cuts in \(\tau\) and the
		size of the formula \(\chi\) are strictly smaller than \(n\), then the size
		of cuts in \(\textsf{push-top}_\chi(\pi, \tau)\) are strictly smaller than 
		\(n\), and
		\item \(\textsf{push-top}\) preserves locality of cuts in witnesses.
	\end{enumerate}
\end{lemma}
\begin{proof}
	We proceed by induction on the pair \((|\chi|, |\pi| + |\tau|)\) (i.e.\ the
	size of \(\chi\) and the sum of the local heights of \(|\pi|\) and \(|\tau|\))
	ordered lexicographically.
	All the possible cut reductions the displayed in Appendix~\ref{sec:cut-reductions}, we are going to argue that the conditions 1 to 3 are fulfilled after each of the reductions, and we have to say what each \(\textsf{cut}_i\) is when it occurs in a cut reduction.

	The case where the label is the cut formula and the axiomatic cases are straightforward using the properties of the auxiliary functions.

	In the principal reduction case, \(\textsf{cut}_1\) and \(\textsf{cut}_2\) are just applications of the I.H.\ (Induction Hypothesis) using that the size of the cut formula is strictly smaller.
	No cut will be in the main local fragment and the condition on the size will be fulfilled since we use the I.H.\ with smaller cut formulas. 
	If the original proof has witnesses with local cuts only, then so does the proof after the cut reduction just by using the I.H. In this argument, we have to use that the auxiliary functions used are strongly preserving.

	For the commutative implication cases, each \(\textsf{cut}_i\) will just be an application of the I.H.\ using that the sum of local heights is smaller with the same cut formula \(\chi\).
	That the conditions are fulfilled can be seen by using that the auxiliary functions used are strongly preserving and employing the I.H.\ in a similar manner to previous case (although for the bound in the sizes of cuts we now use that the I.H.\ is used with \(\chi\), and not with a smaller cut formula).

	We turn to the modal cases.
	If the cut formula is in the weakening part of a modal rule, all the desired properties are easily fulfilled. So we turn to the commutative cases.
	In these, \(\textsf{cut}_i\) will always be a proper application of the cut rule.
	We note that still the main local fragment will have no application of cut since all the \(\text{cut}_i\) appear outside the main local fragment.
	The condition on the sizes of cuts is easily fulfilled since the cut formulas of the \(\textsf{cut}_i\) are smaller or equal than the original cut formula (the cut formula will be of shape \(\nec \chi_0\) or \(\necm \chi_0\) and the \(\text{cut}_i\)'s will have either the cut formula or \(\chi_0\) only).
	We also note that all the auxiliary functions used are at least weakly preserving, so using them will not break this property.

	Lastly we argue for the locality of cuts in witnesses.
	We note one thing: when we apply an auxiliary function to a subproof of the main global fragment of the original proof, we want to use the preservation of locality of cuts in witnesses, while if we apply an auxiliary function to a witness, we want to use the preservation of locality of cuts (because the witness have the cuts local while the subproof will have the property that its witnesses have cuts local).\footnote{
		We remember that witnesses are not subproofs (neither the converse).
	}
	Strongly preserving auxiliary functions preserves both, but change of annotations only preserve locality of cuts.
	However, change of annotations is only applied to witnesses.
	Also, the \(\textsf{cut}_i\)'s we add at witnesses are always in the main local fragment (since we add them at the bottom and we do not add any modal rule at witnesses).

	Let us explicitly discuss what occurs at \(\nec\)-\(\necm_f\) and \(\necm_u\)-\(\necm_f\) reductions, since these cases are a little harder.
	In these cases some witnesses become subproofs after the cut reduction, \(\pi_0\) in the first reduction and both \(\pi_0\) and \(\pi_1\) in the second.
	This is non-problematic for our reasoning since we know that after applying all the auxiliary functions to those witnesses they have local cuts only, and then all their witness have no cuts so in particular all the witnesses have local cuts only.
	In other words, we are adding the cuts in the main local fragment of the witnesses
	to the main global fragment of the proof after the cut reduction, thus enlarging the
	number of cuts in the main global fragment.
	This is non-problematic because the cuts added in this way are outside the 
	main local fragment, which is what we want to have cut-free in this lemma.
\end{proof}

\begin{lemma}\label{lm:cuts-outside-local-fragment}
	There is a function \(\textsf{push-local}\) from proofs in \(\text{G}^\infty_\ell \text{K}^+_s + \textsf{Cut}\) to proofs in \(\text{G}^\infty_\ell \text{K}^+_s + \textsf{Cut}\) such that:
	\begin{enumerate}
		\item for any \(\pi \vdash \Gamma \Rightarrow_s \Delta\) proof in
		\(\text{G}^\infty_\ell \text{K}^+_s + \textsf{Cut}\) we have that
		\(\textsf{push-local}(\pi) \vdash \Gamma \Rightarrow_s \Delta\) with no cuts
		in its main local fragment, 
		\item \(\textsf{push-local}\) preserves sizes of cuts, and
		\item \(\textsf{push-local}\) preserves locality of cuts in witnesses.
	\end{enumerate}
\end{lemma}
\begin{proof}
	The proof is a simple induction on the number of cuts in the main local fragment of \(\pi\), using \textsf{push-top}.
\end{proof}

\subsection{Pushing cuts outside main global fragment}

\begin{lemma}\label{lm:cuts-outside-main-fragment}
	There is a function \(\textsf{push}\) from proofs in 
	\(\text{G}^\infty_\ell \text{K}^+_s + \textsf{Cut}\) to proofs in \(\text{G}^\infty_\ell \text{K}^+_s + \textsf{wCut}\) (so without cuts in its main global fragment) such that:
	\begin{enumerate}
		\item for any \(\pi \vdash \Gamma \Rightarrow_s \Delta\) proof in 
	\(\text{G}^\infty_\ell \text{K}^+_s + \textsf{Cut}\), we have
	\(\textsf{push}(\pi) \vdash \Gamma \Rightarrow_s \Delta\),
		\item \(\textsf{push}\) preserves sizes of cuts, and
		\item \(\textsf{push}\) preserves locality of cuts in witnesses.
	\end{enumerate}
\end{lemma}
\begin{proof}
	To obtain \(\textsf{push}\) it suffices to apply \(\textsf{push-local}\)
	corecursively through the main global fragment.
	By the first condition of \(\textsf{push-local}\) it is clear that the
	resulting proof will have no cuts in its main global fragment since all
	local fragments constituting the new proof are cut-free.

	If we have a bound \(n\) on the size of cuts, by the conditions of
	\(\textsf{push-local}\), we have that each of the local fragments in the new proof
	will have the same bound for cuts.
	But note that if each local fragment of a proof has all the cut-sizes bounded by \(n\), the all cut-sizes of the proof are bounded by \(n\).

	For the condition on witnesses we have the same, by \(\textsf{push-local}\) we know that if we start with a proof whose witnesses have local-cuts only then each local fragment in the resulting proof after the corecursion will have witnesses with local-cuts only.
	But a proof with such local fragments will be a proof whose witnesses have local-cuts only, as desired.
\end{proof}


\section{Cut admissibility}

Our strategy to show cut elimination will require first to show that the rule
\textsf{Cut} is admissible.
Notice that for finitary proofs it is straightforward to prove cut elimination
from cut admissibility by an induction on the height of the proof.
Our cut elimination will follow a similar approach, but it is more involved
since our notion of height is not so straightforward (we do not decrease height
by going to a child node but only by going to a witness).

With cut admissibility in \(\text{G}^\infty_\ell \text{K}^+_s\) we mean that 
for any \(\Gamma, \Delta, \chi\) if 
\(\text{G}^\infty_\ell \text{K}^+_s \vdash \Gamma \Rightarrow_s \Delta, \chi\)
and 
\(\text{G}^\infty_\ell \text{K}^+_s \vdash \chi, \Gamma \Rightarrow_s \Delta\)
then \(\text{G}^\infty_\ell\text{K}^+_s \vdash \Gamma \Rightarrow_s \Delta \).

If we talk about cut admissibility for \(\chi\) we mean the statement that for any \(\Gamma, \Delta\) if \(\text{G}^\infty_\ell \text{K}^+_s \vdash \Gamma \Rightarrow_s \Delta, \chi\) and \(\text{G}^\infty_\ell \text{K}^+_s \vdash \chi, \Gamma \Rightarrow_s \Delta\) then \(\text{G}^\infty_\ell\text{K}^+_s \vdash \Gamma \Rightarrow_s \Delta \), and if we talk about cut admissibility for \(\chi\) with cuts of height smaller than \(\alpha\) we mean the statement that for any \(\Gamma, \Delta, \pi, \tau\) if \(\pi \vdash \Gamma \Rightarrow_s \Delta, \chi\), \(\tau \vdash \chi, \Gamma \Rightarrow_s \Delta\) in \(\text{G}^\infty_\ell \text{K}^+_s\) and \(\|\pi\| \oplus \|\tau\| < \alpha\) (where \(\oplus\) is the Hessenberg sum of ordinals) then \(\text{G}^\infty_\ell \text{K}^+_s \vdash \Gamma \Rightarrow_s \Delta\).

\subsection{Atomic}

\begin{lemma}
	If \(\text{G}^\infty_\ell \text{K}^+_s \vdash \Gamma \Rightarrow_{ s }\Delta, p\) and \(\text{G}^\infty_\ell \text{K}^+_s \vdash p, \Gamma \Rightarrow_{ s }\Delta\), then \(\text{G}^\infty_\ell \text{K}^+_s \vdash \Gamma \Rightarrow_{ s }\Delta\).
\end{lemma}
\begin{proof}
	We prove this by induction on \(|\pi| + |\tau|\) (i.e.\ the sum of local heights).
	The possible cut reductions are the axiomatic cases, the implication commutative cases and the weakening case in the modal cases.
	We note that in all these cases were a cut rule appears in the reduction it can be solved using the I.H.\ (i.e.\ the local height is smaller).
\end{proof}

\subsection{Box formulas}

\begin{lemma}
	If we have cut admissibility for cut formula \(\chi\) and \(\text{G}^\infty_\ell \text{K}^+_s\vdash \Gamma \Rightarrow_{ s }\Delta, \nec \chi\) and \(\text{G}^\infty_\ell \text{K}^+_s\vdash \nec \chi, \Gamma \Rightarrow_{ s }\Delta\), then \(\text{G}^\infty_\ell \text{K}^+_s\vdash \Gamma \Rightarrow_{ s }\Delta\).
\end{lemma}
\begin{proof}
	We prove this by induction on \(|\pi| + |\tau|\) (i.e.\ the sum of local heights).
	The possible cut reductions are the axiomatic cases where the axiomatic character is in a side formula, the implication commutative cases, the weakening case and \(\nec\)-\(\nec\), \(\nec\)-\(\necm_u\) and  \(\nec\)-\(\necm_f\) in the modal cases.
	In the axiomatic and weakening cases, no cut remains after the reduction.
	In the implication commutative cases, we can use the I.H.\ when a \(\textsf{cut}_i\) appears after the reduction since the sum of local heights is smaller.
	Finally, for the  \(\nec\)-\(\nec\), \(\nec\)-\(\necm_u\) and  \(\nec\)-\(\necm_f\)
	cases it suffices to use the assumption of cut admissiblity for \(\chi\).
\end{proof}

\subsection{Master formula}

\begin{definition}
	We say that a proof is \((\necm \chi, \alpha)\)-unblocked iff all its cuts have cut formula \(\chi\) or \(\necm \chi\) and in case the cut formula is \(\necm \chi\) then the height of the cut is \(\leq \alpha\) and the subproofs at the premises are proofs in \(\text{G}^\infty_\ell \text{K}^+_s\) (i.e.\ they have not cuts).
\end{definition}

\begin{lemma}\label{lm:cut-adm-master-top}
	Assume we have cut admissibility for \(\chi\) (with cuts of any height) and also for \(\necm \chi\) with cuts of height smaller than \(\alpha\).
	If \(\pi \vdash \Gamma \Rightarrow_{ s }\Delta, \necm \chi\) and \(\tau \vdash \necm \chi, \Gamma \Rightarrow_{ s }\Delta\) in \(\text{G}^\infty_\ell \text{K}^+_s\) such that \(\|\pi\| \oplus \|\tau\| \leq \alpha\), then there is a \(\rho \vdash \Gamma \Rightarrow_{ s }\Delta\) in \(\text{G}^\infty_\ell \text{K}^+_s + \textsf{mCut}\) 
	that is \((\necm \chi, \alpha)\)-unblocked and has no cuts in its main local fragment.
\end{lemma}
\begin{proof}
	By induction on the sum of the local height of \(\pi\) and \(\tau\).
	The possible cut reductions are the axiomatic cases where the label is the cut formula, the axiomatic character is in a side formula, the commutative implication cases and the modal cases of weakening or \(\necm_u\)-\(\nec\), \(\necm_u\)-\(\necm_u\) and \(\necm_u\)-\(\necm_f\).
	In the reduction rules where the label is the cut formula, axiomatic and modal weakening no cuts remain, so clearly the desired result holds.

	In the commutative implication reductions it suffices to substitute each \(\textsf{cut}_i\) with an application of the I.H.

	Finally, we treat the \(\necm_u\)-\(\nec\), \(\necm_u\)-\(\necm_u\) and \(\necm_u\)-\(\necm_f\) reductions.
	By inspection we see that all the main local fragments have no cuts.
	Finally, let us see how to interpret \(\textsf{cut}_i\) in each case to get a proof with main cuts only and the unblocked condition.

	\(\necm_u\)-\(\nec\). In this case \(\textsf{cut}_1\) is admissibility of cuts with cut formula \(\necm \chi\) and height smaller than \(\alpha\) and \(\textsf{cut}_2\) is admissibility of cuts with cut formula \(\chi\). We are allowed to used the first admissibility for the \(\textsf{cut}_1\) since
	\begin{multline*}
	  \|{\textsf{wk}(\pi_1)}^\circ\| \oplus \| \textsf{wk}(\tau_0) \| \leq 	(\|{\textsf{wk}(\pi_1)} \| + 1) \oplus \| \textsf{wk}(\tau_0) \| \leq \\  (\|\pi_1\| + 1) \oplus \| \tau_0 \| \leq \| \pi \| \oplus \| \tau_0 \| < \| \pi \| \oplus \|\tau\| = \alpha
	\end{multline*}
	thanks to \(\|\pi_1\| < \|\pi\|, \|\tau_0\| < \|\tau\|\) being witnesses.
	This gives a proof without cuts so it has main cuts only and it is unblocked, as desired.

	\(\necm_u\)-\(\nec_u\). In this case \(\textsf{cut}_1, \textsf{cut}_3\) are by admissibility of cuts with cut formula \(\necm \chi\) and height smaller than \(\alpha\) and \(\textsf{cut}_2, \textsf{cut}_4\) is admissibility of cuts with cut formula \(\chi\). We are allowed to used the first admissibility for \(\textsf{cut}_1\) and \(\textsf{cut}_3\) since
	\begin{multline*}
	  \|{\textsf{wk}(\pi_1)}^\circ\| \oplus \| \textsf{wk}(\tau_0) \| \leq 	(\|{\textsf{wk}(\pi_1)} \| + 1) \oplus \| \textsf{wk}(\tau_0) \| \leq \\  (\|\pi_1\| + 1) \oplus \| \tau_0 \| \leq \| \pi \| \oplus \| \tau_0 \| < \| \pi \| \oplus \|\tau\| = \alpha
	\end{multline*}
	\begin{multline*}
	  \|{\textsf{wk}(\pi_1)}^\phi\| \oplus \| \textsf{wk}(\tau_1) \| \leq 	(\|{\textsf{wk}(\pi_1)} \| + 1) \oplus \| \textsf{wk}(\tau_1) \| \leq \\  (\|\pi_1\| + 1) \oplus \| \tau_1 \| \leq \| \pi \| \oplus \| \tau_1 \| < \| \pi \| \oplus \|\tau\| = \alpha
	\end{multline*}
	thanks to \(\|\pi_1\| < \|\pi\|, \|\tau_0\|, \|\tau_1\| < \|\tau\|\) being witnesses.
	This gives a proof without cuts so it has main cuts only and it is unblocked, as desired.
	
	\(\necm_u\)-\(\nec_f\). In this case \(\textsf{cut}_1\) is by admissibility of cuts with cut formula \(\necm \chi\) and height smaller than \(\alpha\) and \(\textsf{cut}_2\) is by admissibility of cuts with cut formula \(\chi\) and \(\textsf{cut}_3, \textsf{cut}_4\) are standard cuts. We are allowed to used the first admissibility for \(\textsf{cut}_1\) since
	\begin{multline*}
	  \|{\textsf{wk}(\pi_1)}^\circ\| \oplus \| \textsf{wk}(\tau_0) \| \leq 	(\|{\textsf{wk}(\pi_1)} \| + 1) \oplus \| \textsf{wk}(\tau_0) \| \leq \\  (\|\pi_1\| + 1) \oplus \| \tau_0 \| \leq \| \pi \| \oplus \| \tau_0 \| < \| \pi \| \oplus \|\tau\| = \alpha
	\end{multline*}
	thanks to \(\|\pi_1\| < \|\pi\|, \|\tau_0\| < \|\tau\|\) being witnesses.
	There are two cuts after the reduction, one with cut formula \(\chi\) (\(\textsf{cut}_4\)) and other with cut formula \(\necm \chi\) (\(\textsf{cut}_3\)).
	\(\textsf{cut}_3\) has the premises with no cuts, and its height is
	\begin{multline*}
	  \|{\textsf{wk}(\pi_1)}^\phi\| \oplus \| \textsf{wk}(\tau_1) \| \leq 	(\|{\textsf{wk}(\pi_1)} \| + 1) \oplus \| \textsf{wk}(\tau_1) \| \leq \\  (\|\pi_1\| + 1) \oplus \| \tau_1 \| \leq \| \pi \| \oplus \| \tau_1 \| \leq \| \pi \| \oplus \|\tau\| = \alpha
	\end{multline*}
	thanks to \(\|\pi_1\| < \|\pi\|\) being a witness and \(\|\tau_1\| \leq \|\tau\|\) being a subproof.
	This means that the resulting proof has only main cuts and it is \((\necm \chi, \alpha)\)-unblocked.
\end{proof}

\begin{lemma}\label{lm:cut-adm-master-local}
	Assume we have cut admissibility for \(\chi\) (with cuts of any height) and also for \(\necm \chi\) with cuts of height smaller than \(\alpha\).
	Let \(\pi \vdash \Gamma \Rightarrow_s \Delta\) in \(\text{G}^\infty_\ell \text{K}^+_s + \textsf{mCut}\) and assume it is \((\necm \chi, \alpha)\)-unblocked.
	Then, there is a \(\rho \vdash \Gamma \Rightarrow_s \Delta\) in \(\text{G}^\infty_\ell \text{K}^+_s + \textsf{mCut}\) that is \((\necm \chi, \alpha)\)-unblocked and has no cuts with cut formula \(\necm \chi\) in its main local fragment.
\end{lemma}
\begin{proof}
	This is a simple induction using Lemma~\ref{lm:cut-adm-master-top} in the
	number of cuts with cut formula \(\necm \chi\) in its main local fragment.
	The assumption that \(\pi\) is \((\necm \chi, \alpha)\)-unblocked is of
	fundamental importance since it allows use to use
	Lemma~\ref{lm:cut-adm-master-top} to any cut at the main local fragment with
	cut formula \(\necm \chi\).
\end{proof}

\begin{lemma}\label{lm:cut-adm-master-aux}
	Assume we have cut admissibility for \(\chi\) (with cuts of any height) and also for \(\necm \chi\) with cuts of height smaller than \(\alpha\).	
	Then \(\text{G}^\infty_\ell \text{K}^+_s \vdash \Gamma \Rightarrow_{ s }\Delta, \necm \chi\) and \(\text{G}^\infty_\ell \text{K}^+_s \vdash \necm \chi, \Gamma \Rightarrow_{ s }\Delta\), implies \(\text{G}^\infty_\ell \text{K}^+_s + \textsf{mCut} \vdash \Gamma \Rightarrow_{ s }\Delta\) such that all the cuts have \(\chi\) as cut formula.
\end{lemma}
\begin{proof}
	This is just a corecursion for a local progressing system (i.e.\ at each
	corecursive step we provide a whole local fragment, not just a node) using
	Lemma~\ref{lm:cut-adm-master-local}.
\end{proof}

\begin{lemma}
	Assume we have cut admissibility for cuts of size smaller than \(n\) and let \(|\necm \chi| = n\).
	Then \(\text{G}^\infty_\ell \text{K}^+_s \vdash \Gamma \Rightarrow_{ s }\Delta, \necm \chi\) and \(\text{G}^\infty_\ell \text{K}^+_s \vdash \necm \chi, \Gamma \Rightarrow_{ s }\Delta\) implies \mbox{\(\text{G}^\infty_\ell \text{K}^+_s \vdash \Gamma \Rightarrow_{ s }\Delta\)}.
\end{lemma}
\begin{proof}
	Let \(\pi \vdash \Gamma \Rightarrow_s \Delta, \necm \chi\) and \(\tau \vdash \necm \chi, \Gamma \Rightarrow_s \Delta\).
	We proceed by induction on \(\|\pi\| \oplus \|\tau\|\), so assume we have cut admissibility for \(\necm\chi\) with cuts of height smaller than
	\(\|\pi\| \oplus \|\tau\|\).
	We also have cut admisibility for \(\chi\) thanks to the assumptions.
	Then by Lemma~\ref{lm:cut-adm-master-aux} we can get a proof \(\rho_0\) of the
	same sequent in \(\text{G}^\infty_\ell\text{K}^+_s + \textsf{mCut}\) whose
	only cut formula is \(\chi\).

	Then we can use Lemma~\ref{lm:cuts-outside-main-fragment} to obtain a proof 
	\(\rho_1\) with witness cuts only and all of them of size smaller than \(n\)
	and each ocurring at the main local fragment of some witness (so each witness have
	a finite amount of cuts).
	By an induction on the number of local cuts in the witness and using the 
	hypothesis of cut admissibility for size \(< n\) we can change each witness 
	to a cut-free witness, obtaining the desired proof.
\end{proof}

\subsection{General case}

\begin{theorem}[Cut admissibility]\label{th:cut-admissibility}
	Let \(\text{G}^\infty_\ell \text{K}^+_s \vdash \Gamma \Rightarrow_s \Delta, \chi\) and \(\text{G}^\infty_\ell \text{K}^+_s \vdash \chi, \Gamma \Rightarrow_s \Delta\).
	Then, we have that \(\text{G}^\infty_\ell \text{K}^+_s \vdash \Gamma \Rightarrow_s \Delta\).
\end{theorem}
\begin{proof}
	This is a simple induction on the size of \(\chi\) using the lemmas previously proved in this section.
	In case \(\chi\) is an implication, i.e.\ of shape \(\chi_0 \to \chi_1\), first we need to apply the inversion to get formulas of smaller size (and weakening to make the sequents match).
\end{proof}

With cut admissibility we can show cut elimination for proofs with finitely many cuts.
However, in the non-wellfounded setting it is not straightforward that this gives us cut elimination for any proof, since our proofs may have infinitely many cuts.
The purpose of the next section will be to show cut elimination in full generality for \(\text{G}^\infty_\ell \text{K}^+_s\) using cut admissibility, thus providing and example of how to get cut elimination from admissibility in the non-wellfounded setting.
Before finalizing this section we show a corollary which we will use during cut elimination.

\begin{corollary}\label{cor:cut-admissibility}
	Suppose that \(\pi \vdash \Gamma \Rightarrow_s \Delta\) in \(\text{G}^\infty_\ell \text{K}^+_s + \textsf{Cut}\).
	Assume that either:
	\begin{enumerate}
		\item \(\pi\) has finitely many instances of cut, or
		\item \(\pi\) has local cuts only.
	\end{enumerate}
	Then \(\text{G}^\infty_\ell \text{K}^+_s \vdash \Gamma \Rightarrow_s \Delta \).
\end{corollary}
\begin{proof}
	In case \(\pi\) has local cuts only then it must be the case that \(\pi\) has finitely many instances of cut, since the main local fragment is always finite.
	Then, in both cases we can assume that \(\pi\) has finitely many instances of
	cut.
	The result is then proven by an induction on the number of cuts using
	Theorem~\ref{th:cut-admissibility}.
\end{proof}

\section{Cut elimination}

We finish the paper by proving the promised result: cut elimination for \(\text{G}^\infty_\ell \text{K}^+_s\).
Note that by the translations defined in Section~\ref{sec:definition} between \(\text{G}^\infty_\ell \text{K}^+\) and \(\text{G}^\infty_\ell \text{K}^+_s\) this cut elimination will prove cut elimination for \(\text{G}^\infty_\ell \text{K}^+\).
In fact, since in reality the difference between \(\text{G}^\infty_\ell \text{K}^+_s\) and \(\text{G}^\infty_\ell \text{K}^+\) is just how we accomodate the information of proofs, it can be argued that this cut elimination is a cut elimination method for \(\text{G}^\infty_\ell \text{K}^+\) and \(\text{G}^\infty_\ell \text{K}^+_s\) just provides a way to define the necessary corecursive functions easier.

\begin{theorem}[Cut Elimination]\label{th:cut-elimination}
	If \(\text{G}^\infty_\ell \text{K}^+_s + \textsf{Cut} \vdash \Gamma \Rightarrow_s \Delta\), then  \(\text{G}^\infty_\ell \text{K}^+_s \vdash \Gamma \Rightarrow_s \Delta\).
\end{theorem}   
\begin{proof}
	Let \(\pi \vdash \Gamma \Rightarrow_s \Delta\) in \(\text{G}^\infty_\ell \text{K}^+_s + \textsf{Cut}\).
	We proceed by induction on \(\|\pi\|\).
	Note that for any witness \(\tau\) of \(\pi\) we have that \(\|\tau\| < \|\pi\|\).
	So by induction hypothesis we can get a \(\tau'\) proof in \(\text{G}^\infty_\ell \text{K}^+_s\) proving the same sequent as \(\tau\).
	Let \(\pi_1\) be the result of replacing each of its witnesses \(\tau\) by \(\tau'\), we notice that \(\pi_1\) is a proof in \(\text{G}^\infty_\ell \text{K}^+_s + \textsf{mCut}\) proving the same sequent as \(\pi\) (changing witnesses do not alter the conclusion, but we may have altered the height of the proof so from now on the I.H.\ cannot be used anymore).

	Let \(\pi_2 := \textsf{push}(\pi_1)\), by Lemma~\ref{lm:cuts-outside-main-fragment} (using that a proof with main cuts only have witnesses with local cuts only) we have that \(\pi_2\) is proof in \(\text{G}^\infty_\ell \text{K}^+_s + \textsf{wCut}\) proving the same sequent as \(\pi\) and such that all its witnesses have local cuts only.

	Let \(\iota\) be a witness of \(\pi_2\), by Corollary~\ref{cor:cut-admissibility} there is an \(\iota'\) proof in \(\text{G}^\infty_\ell \text{K}^+_s\) proving the same sequent as \(\iota\).
	Let \(\pi_3\) be the result of replacing each witness \(\iota\) of \(\pi_2\) by \(\iota'\). Then \(\pi_3\) proves the same sequent as \(\pi\), since a change of witnesses do not change the conclusion. In addition since \(\pi_2\) have witness-cuts only and \(\pi_3\) has the same main global fragment as \(\pi_2\) but with no cuts in the witnesses we can conclude that \(\pi_3\) is cut free.
\end{proof}

\begin{figure}
	\centering
    \begin{tikzpicture}
      \filldraw[gray!50] (0,0) -- (2,1) -- (-2,1) -- cycle;
      \draw (0,0) -- (2,1) -- (-2,1) -- cycle;

      \filldraw[gray!50] (-1.5,1) -- (-1,3) -- (-2,3) -- cycle;
      \draw (-1.5,1) -- (-1,3) -- (-2,3) -- cycle;
      \filldraw[gray!50] (1.5,1) -- (1,3) -- (2,3) -- cycle;
      \draw (1.5,1) -- (1,3) -- (2,3) -- cycle;
      \node at (0,2) {\(\ldots\)};
      \node at (3,1) {\LARGE\(\rightsquigarrow\)};
    \end{tikzpicture}
    \quad
    \begin{tikzpicture}
      \filldraw[gray!50] (0,0) -- (2,1) -- (-2,1) -- cycle;
      \draw (0,0) -- (2,1) -- (-2,1) -- cycle;

      \draw (-1.5,1) -- (-1,3) -- (-2,3) -- cycle;
      \draw (1.5,1) -- (1,3) -- (2,3) -- cycle;
      \node at (0,2) {\(\ldots\)};
      \node at (3,1) {\LARGE\(\rightsquigarrow\)};
    \end{tikzpicture}
    \vspace{0.5cm}

    \begin{tikzpicture}
      \node at (-3,1) {\LARGE\(\rightsquigarrow\)};
      \draw (0,0) -- (2,1) -- (-2,1) -- cycle;

      \filldraw[gray!50] (-1.5,1) -- (-1.375,1.5) -- (-1.625,1.5) -- cycle;
      \draw (-1.5,1) -- (-1,3) -- (-2,3) -- cycle;
      \filldraw[gray!50] (1.5,1) -- (1.375,1.5) -- (1.625,1.5) -- cycle;
      \draw (1.5,1) -- (1,3) -- (2,3) -- cycle;
      \node at (0,2) {\(\ldots\)};
      \node at (3,1) {\LARGE\(\rightsquigarrow\)};
    \end{tikzpicture}
    \begin{tikzpicture}
      \draw (0,0) -- (2,1) -- (-2,1) -- cycle;

      \draw (-1.5,1) -- (-1,3) -- (-2,3) -- cycle;
      \draw (1.5,1) -- (1,3) -- (2,3) -- cycle;
      \node at (0,2) {\(\ldots\)};
    \end{tikzpicture}
    \caption{A graphical representation of the cut elimination process.}
  	\label{fig:cut-elim}
\end{figure}

  We can see a graphical representation of the process below at Figure~\ref{fig:cut-elim}.
The wide triangle represents the global fragment of the proof, while the thinner triangles represents the witness.
We note that there could be infinitely many witnesses, which is represented by the horizontal dots.
We start with all the triangles being gray, which represents that cut can occur at all levels: in the main fragment and in the main global fragment.
In the first step we apply the induction hypothesis to all witnesses at the same time, meaning that all the cuts left are in the main global fragment.
This is represented in the picture by making the witnesses triangles white.
In the second step we apply the \(\textsf{push}\) function.
Since after the first step the proof had witnesses with local cuts only, after the second step this property is preserved.
This is represented by having the witness triangles with a little bit of gray, but not fully covered in gray.
On the other hand, thanks to the application of \(\textsf{push}\) we know that the main global fragment is cut free.
Then all the cuts which are left are in the witnesses and occur in the main local fragment of them.
In the third step, we simultaneously apply cut admissibility on all the witnesses (since the main local fragment is finite having all the cuts there implies that there are only finitely many cuts) in order to obtain a cut-free proof, as desired.

We would like to finish the paper with a remark.
When defining cut elimination methods for non-wellfounded proofs, the cut-free proof is usually defined via the convergence of its approximations (e.g. \cite{acclavio2024infinitarycuteliminationfiniteapproximations}).
When a method of this kind is applied to a preproof, it will not necessarily converge.
In our case we face a similar problem with preproofs, namely that they lack the notions of local height and ordinal height.
The recursions we use would not be available and what to do would not even be defined.
So for our methodology slicing the proof into adequate parts, which we call local and global fragments, is fundamental.
The slices help us to organize the order in which we push the cuts, requiring some back-and-forth, another thing that contrasts with other cut eliminations methods for non-wellfounded proofs.\footnote{We would like to thank an anonymous reviewer for this observation.}

\section*{Conclusion and future work}

We proved cut elimination for a non-wellfounded calculus of \(\text{K}^+\), using only basic techniques of structural proof theory, such as ordinal recursion and corecursion.
The method is mainly based in splitting the proofs nicely, taming the global branch condition in the process.
This proof also works for similar systems to the master modality, such as common knowledge over \(\mathrm{K}\), in which the \(\nec\) and \(\necm\) rules change to (in the system \(\text{G}^\infty_\ell\text{K}^+\)):
\[
  \AxiomC{}
  \noLine
  \UnaryInfC{\(\Gamma, \Pi, \mathsf{C} \Pi \Rightarrow_{\circ} \phi\)}
  \RightLabel{\(\nec_i\)}
  \UnaryInfC{\(\Sigma, \nec_i \Gamma, \mathsf{C} \Pi \Rightarrow_{s} \nec_i\phi, \Delta\)}
  \DisplayProof
  \qquad
  \AxiomC{}
  \noLine
  \UnaryInfC{\({(\Gamma_i, \Pi, \mathsf{C} \Pi \Rightarrow_{\circ} \phi)}_{i < n}\)}
  \AxiomC{}
  \noLine
  \UnaryInfC{\({(\Gamma_i, \Pi, \mathsf{C} \Pi \Rightarrow_{\phi} \mathsf{C} \phi)}_{i<n}\)}
  \RightLabel{\(C\)}
  \BinaryInfC{\(\Sigma, \nec_0 \Gamma_0, \ldots, \nec_{n-1} \Gamma_{n-1}, \mathsf{C} \Pi
  \Rightarrow_{s} \mathsf{C} \phi, \Delta \)}
  \DisplayProof
\]
where we assume we have \(n\) agents, \(i \in \{0, \ldots, n-1\}\),
\(\nec_i\) is the knowledge modality for agent \(i\), and \(\mathsf{C}\) is the common knowledge modality.


In \cite{previous} proofs consists of a base recursive level with a corecursive level on top. In the present article proofs have those two levels and on top of them another recursive level, an extension which is needed to acommodate the master modality.
This suggest a hierarchy of classes of proof systems in which each time the number recursive-corecursive alternation increases, allowing greater expressivity.
Exploring the intricacies of this hierarchy is left as future work.

Another line of work is to explore how our method can be extended to more complex branch conditions
such as the ones used to deal with the modal \(\mu\)-calculus.
In this direction there are two questions which we consider of interest to explore:
\begin{enumerate}
  \item How can we slice non-wellfounded proofs to ease cut eliminition?
  \item Which methods can be use to show cut elimination from  cut admissibility
    for non-wellfounded proofs?
\end{enumerate}


\appendix

\section{Weakening}

\begin{definition}
Given \(\pi \vdash \Gamma \Rightarrow_{ s }\Delta\) in \(\text{G}^{ \infty }_{ \ell }\text{K}^{ + }_{ s }\) we can define, by recursion on the local height and cases in the last rule apply, the proof \(\textsf{wk}_{ \Gamma';\Delta' }(\pi)\) proof of \(\Gamma, \Gamma' \Rightarrow_{ s }\Delta,\Delta'\) in \(\text{G}^{ \infty }_{ \ell }\text{K}^{ + }_{ s }\) as:

\[
\AxiomC{}
\RightLabel{Ax}
\UnaryInfC{\(\Gamma, p \Rightarrow_s p, \Delta\)}
\DisplayProof
\mapsto
\AxiomC{}
\RightLabel{Ax}
\UnaryInfC{\(\Gamma, \Gamma', p \Rightarrow_{s} p, \Delta, \Delta'\)}
\DisplayProof
\]

\[
\AxiomC{}
\RightLabel{Ax-\(\bot\)}
\UnaryInfC{\(\Gamma, \bot \Rightarrow_s \Delta\)}
\DisplayProof
\mapsto
\AxiomC{}
\RightLabel{Ax-\(\bot\)}
\UnaryInfC{\(\Gamma, \Gamma'\bot \Rightarrow_{s} \Delta, \Delta'\)}
\DisplayProof
\]

\[
\AxiomC{\(\pi_{ 0 }\)}
\noLine
\UnaryInfC{\(\Gamma \Rightarrow_s \Delta, \phi\)}
\AxiomC{\(\pi_{ 1 }\)}
\noLine
\UnaryInfC{\(\Gamma, \psi \Rightarrow_s \Delta\)}
\RightLabel{\(\to\)L}
\BinaryInfC{\(\Gamma, \phi \to \psi \Rightarrow_s \Delta\)}
\DisplayProof
\mapsto
\AxiomC{\(\textsf{wk}_{ \Gamma';\Delta' }(\pi_{ 0})\)}
\noLine
\UnaryInfC{\(\Gamma, \Gamma' \Rightarrow_{ s} \Delta, \Delta', \phi\)}
\AxiomC{\(\textsf{wk}_{ \Gamma';\Delta' }(\pi_{ 1 })\)}
\noLine
\UnaryInfC{\(\Gamma, \Gamma', \psi \Rightarrow_{s} \Delta, \Delta'\)}
\RightLabel{\(\to\)L}
\BinaryInfC{\(\Gamma, \Gamma', \phi \to \psi \Rightarrow_{s} \Delta, \Delta'\)}
\DisplayProof
\]

\[
\AxiomC{\(\pi_{ 0 }\)}
\noLine
\UnaryInfC{\(\Gamma, \phi \Rightarrow_s \psi, \Delta\)}
\RightLabel{\(\to\)R}
\UnaryInfC{\(\Gamma \Rightarrow_s \phi \to \psi, \Delta\)}
\DisplayProof
\mapsto
\AxiomC{\(\textsf{wk}_{ \Gamma';\Delta' }(\pi_{ 0})\)}
\noLine
\UnaryInfC{\(\Gamma, \Gamma', \phi \Rightarrow_{s} \psi, \Delta, \Delta'\)}
\RightLabel{\(\to\)R}
\UnaryInfC{\(\Gamma, \Gamma' \Rightarrow_{s} \phi \to \psi, \Delta, \Delta'\)}
\DisplayProof
\]

\[
\AxiomC{\(\tau \vdash \Gamma, \dnecm \Pi \Rightarrow_\circ \phi\)}
\RightLabel{\(\nec\)}
\UnaryInfC{\(\Sigma, \nec \Gamma, \necm \Pi \Rightarrow_s \nec \phi, \Delta\)}
\DisplayProof
\mapsto 
\AxiomC{\(\tau \vdash \Gamma, \dnecm \Pi \Rightarrow_\circ \phi\)}
\RightLabel{\(\nec\)}
\UnaryInfC{\(\Sigma, \Gamma', \nec \Gamma, \necm \Pi \Rightarrow_{s} \nec \phi, \Delta, \Delta'\)}
\DisplayProof
\]

If \(\pi\) is:
\[
\AxiomC{\(\tau \vdash \Gamma, \dnecm \Pi \Rightarrow_\circ \phi\)}
\AxiomC{\(\pi_{ 0 }\)}
\noLine
\UnaryInfC{\(\Gamma, \dnecm \Pi \Rightarrow_\phi \necm \phi\)}
\RightLabel{\(\necm_f\)}
\BinaryInfC{\(\Sigma, \nec \Gamma, \necm \Pi \Rightarrow_s \necm \phi, \Delta\)}
\DisplayProof
\]
then it maps to

\[
\AxiomC{\(\tau \vdash^{ \beta } \Gamma, \dnecm \Pi \Rightarrow_\circ \phi\)}
\AxiomC{\(\pi_{ 0 }\)}
\noLine
\UnaryInfC{\(\Gamma, \dnecm \Pi \Rightarrow_\phi \necm \phi\)}
\RightLabel{\(\necm_f\)}
\BinaryInfC{\(\Sigma, \Gamma', \nec \Gamma, \necm \Pi \Rightarrow_{s} \necm \phi, \Delta, \Delta'\)}
\DisplayProof
\]

If \(\pi\) is:

\[
\AxiomC{\(\tau_0 \vdash \Gamma, \dnecm \Pi \Rightarrow_\circ \phi\)}
\AxiomC{\(\tau_1 \vdash \Gamma, \dnecm \Pi \Rightarrow_\phi \necm \phi\)}
\RightLabel{\(\necm_u\)}
\BinaryInfC{\(\Sigma, \nec \Gamma, \necm \Pi \Rightarrow_s \necm \phi, \Delta\)}
\DisplayProof
\]

then it maps to

\[
\AxiomC{\(\tau_0 \vdash \Gamma, \dnecm \Pi \Rightarrow_\circ \phi\)}
\AxiomC{\(\tau_1 \vdash \Gamma, \dnecm \Pi \Rightarrow_\phi \necm \phi\)}
\RightLabel{\(\necm_u\)}
\BinaryInfC{\(\Sigma, \Gamma', \nec \Gamma, \necm \Pi \Rightarrow_{s'} \necm \phi, \Delta, \Delta'\)}
\DisplayProof
\]

\[
\AxiomC{\(\pi_{ 0 }\)}
\noLine
\UnaryInfC{\(\Gamma \Rightarrow_{ s } \Delta, \chi\)}
\AxiomC{\(\pi_{ 1 }\)}
\noLine
\UnaryInfC{\(\chi, \Gamma \Rightarrow_{ s }\Delta\)}
\RightLabel{Cut}
\BinaryInfC{\(\Gamma \Rightarrow_{ s }\Delta\)}
\DisplayProof
\mapsto
\AxiomC{\(\textsf{wk}_{ \Gamma';\Delta' }(\pi_{ 0})\)}
\noLine
\UnaryInfC{\(\Gamma, \Gamma' \Rightarrow_{ s } \Delta, \Delta', \chi\)}
\AxiomC{\(\textsf{wk}_{ \Gamma';\Delta' }(\pi_{ 1})\)}
\noLine
\UnaryInfC{\(\chi, \Gamma, \Gamma' \Rightarrow_{ s }\Delta, \Delta'\)}
\RightLabel{Cut}
\BinaryInfC{\(\Gamma, \Gamma' \Rightarrow_{ s }\Delta, \Delta'\)}
\DisplayProof
\]
\end{definition}

\begin{lemma}
The function \(\textsf{wk}_{ \Gamma';\Delta' }\) is strongly preserving.
\end{lemma}

\section{Contraction}
\subsection{Contraction of propositional variable in the left}

\begin{definition}
Given \(\pi \vdash \Gamma,p,p \Rightarrow_{ s }\Delta\) in \(\text{G}^{ \infty }_{ \ell }\text{K}^{ + }_{ s }\) we can define, by recursion on the local height and cases in the last rule apply, the proof \(\textsf{lctr}_{ p }(\pi)\) proof of \(\Gamma, p \Rightarrow_{ s }\Delta\) in \(\text{G}^{ \infty }_{ \ell }\text{K}^{ + }_{ s }\) as:

\[
\AxiomC{}
\RightLabel{Ax}
\UnaryInfC{\(\Gamma, p, p \Rightarrow_s p, \Delta\)}
\DisplayProof
\mapsto
\AxiomC{}
\RightLabel{Ax}
\UnaryInfC{\(\Gamma, p \Rightarrow_{s} p, \Delta \)}
\DisplayProof
\]

\[
\AxiomC{}
\RightLabel{Ax}
\UnaryInfC{\(\Gamma, p, p, q  \Rightarrow_s q, \Delta\)}
\DisplayProof
\mapsto
\AxiomC{}
\RightLabel{Ax}
\UnaryInfC{\(\Gamma, p, q \Rightarrow_{s} q, \Delta \)}
\DisplayProof
\]
for \(q \neq p\).

\[
\AxiomC{}
\RightLabel{Ax-\(\bot\)}
\UnaryInfC{\(\Gamma,p,p,\bot \Rightarrow_s \Delta\)}
\DisplayProof
\mapsto
\AxiomC{}
\RightLabel{Ax-\(\bot\)}
\UnaryInfC{\(\Gamma,p,\bot \Rightarrow_{s} \Delta\)}
\DisplayProof
\]

\[
\AxiomC{\(\pi_{ 0 }\)}
\noLine
\UnaryInfC{\(\Gamma,p,p \Rightarrow_s \Delta, \phi\)}
\AxiomC{\(\pi_{ 1 }\)}
\noLine
\UnaryInfC{\(\Gamma,p,p, \psi \Rightarrow_s \Delta\)}
\RightLabel{\(\to\)L}
\BinaryInfC{\(\Gamma,p,p, \phi \to \psi \Rightarrow_s \Delta\)}
\DisplayProof
\mapsto
\AxiomC{\(\textsf{lctr}_{ p }(\pi_{ 0})\)}
\noLine
\UnaryInfC{\(\Gamma, p \Rightarrow_{ s} \Delta, \phi\)}
\AxiomC{\(\textsf{lctr}_{ p }(\pi_{ 1 })\)}
\noLine
\UnaryInfC{\(\Gamma, p, \psi \Rightarrow_{s} \Delta\)}
\RightLabel{\(\to\)L}
\BinaryInfC{\(\Gamma, p, \phi \to \psi \Rightarrow_{s} \Delta\)}
\DisplayProof
\]

\[
\AxiomC{\(\pi_{ 0 }\)}
\noLine
\UnaryInfC{\(\Gamma, p, p, \phi \Rightarrow_s \psi, \Delta\)}
\RightLabel{\(\to\)R}
\UnaryInfC{\(\Gamma, p, p \Rightarrow_s \phi \to \psi, \Delta\)}
\DisplayProof
\mapsto
\AxiomC{\(\textsf{lctr}_{ p }(\pi_{ 0})\)}
\noLine
\UnaryInfC{\(\Gamma, p, \phi \Rightarrow_{s} \psi, \Delta\)}
\RightLabel{\(\to\)R}
\UnaryInfC{\(\Gamma, p \Rightarrow_{s} \phi \to \psi, \Delta\)}
\DisplayProof
\]

\[
\AxiomC{\(\tau \vdash \Gamma, \dnecm \Pi \Rightarrow_\circ \phi\)}
\RightLabel{\(\nec\)}
\UnaryInfC{\(\Sigma,p,p, \nec \Gamma, \necm \Pi \Rightarrow_s \nec \phi, \Delta\)}
\DisplayProof
\mapsto 
\AxiomC{\(\tau \vdash \Gamma, \dnecm \Pi \Rightarrow_\circ \phi\)}
\RightLabel{\(\nec\)}
\UnaryInfC{\(\Sigma, p, \nec \Gamma, \necm \Pi \Rightarrow_{s} \nec \phi, \Delta\)}
\DisplayProof
\]

If \(\pi\) is:
\[
\AxiomC{\(\tau \vdash \Gamma, \dnecm \Pi \Rightarrow_\circ \phi\)}
\AxiomC{\(\pi_{ 0 }\)}
\noLine
\UnaryInfC{\(\Gamma, \dnecm \Pi \Rightarrow_\phi \necm \phi\)}
\RightLabel{\(\necm_f\)}
\BinaryInfC{\(\Sigma,p,p,\nec \Gamma, \necm \Pi \Rightarrow_s \necm \phi, \Delta\)}
\DisplayProof
\]
then it maps to

\[
\AxiomC{\(\tau \vdash \Gamma, \dnecm \Pi \Rightarrow_\circ \phi\)}
\AxiomC{\(\pi_{ 0 }\)}
\noLine
\UnaryInfC{\(\Gamma, \dnecm \Pi \Rightarrow_\phi \necm \phi\)}
\RightLabel{\(\necm_f\)}
\BinaryInfC{\(\Sigma,p, \nec \Gamma, \necm \Pi \Rightarrow_{s} \necm \phi, \Delta\)}
\DisplayProof
\]

If \(\pi\) is:

\[
\AxiomC{\(\tau_0 \vdash \Gamma, \dnecm \Pi \Rightarrow_\circ \phi\)}
\AxiomC{\(\tau_1 \vdash \Gamma, \dnecm \Pi \Rightarrow_\phi \necm \phi\)}
\RightLabel{\(\necm_u\)}
\BinaryInfC{\(\Sigma,p,p, \nec \Gamma, \necm \Pi \Rightarrow_s \necm \phi, \Delta\)}
\DisplayProof
\]

then it maps to

\[
\AxiomC{\(\tau_0 \vdash \Gamma, \dnecm \Pi \Rightarrow_\circ \phi\)}
\AxiomC{\(\tau_1 \vdash \Gamma, \dnecm \Pi \Rightarrow_\phi \necm \phi\)}
\RightLabel{\(\necm_u\)}
\BinaryInfC{\(\Sigma, p, \nec \Gamma, \necm \Pi \Rightarrow_{s} \necm \phi, \Delta\)}
\DisplayProof
\]

\[
\AxiomC{\(\pi_{ 0 }\)}
\noLine
\UnaryInfC{\(\Gamma,p,p \Rightarrow_{ s } \Delta, \chi\)}
\AxiomC{\(\pi_{ 1 }\)}
\noLine
\UnaryInfC{\(\chi, \Gamma,p,p \Rightarrow_{ s }\Delta\)}
\RightLabel{Cut}
\BinaryInfC{\(\Gamma,p,p \Rightarrow_{ s }\Delta\)}
\DisplayProof
\mapsto
\AxiomC{\(\textsf{lctr}_{ p }(\pi_{ 0})\)}
\noLine
\UnaryInfC{\(\Gamma, p \Rightarrow_{ s } \Delta, \chi\)}
\AxiomC{\(\textsf{lctr}_{ p }(\pi_{ 1})\)}
\noLine
\UnaryInfC{\(\chi, \Gamma, p \Rightarrow_{ s }\Delta\)}
\RightLabel{Cut}
\BinaryInfC{\(\Gamma, p \Rightarrow_{ s }\Delta \)}
\DisplayProof
\]
\end{definition}

\begin{lemma}
The function \(\textsf{lctr}_{ p }\) is strongly preserving.
\end{lemma}

\subsection{Contraction of propositional variable in the right}

\begin{definition}
Given \(\pi \vdash \Gamma \Rightarrow_{ s }\Delta,p,p\) in \(\text{G}^{ \infty }_{ \ell }\text{K}^{ + }_{ s }\) we can define, by recursion on the local height and cases in the last rule apply, the proof \(\textsf{lctr}_{ p }(\pi)\) proof of \(\Gamma \Rightarrow_{ s } \Delta,p\) in \(\text{G}^{ \infty }_{ \ell }\text{K}^{ + }_{ s }\) as:

\[
\AxiomC{}
\RightLabel{Ax}
\UnaryInfC{\(\Gamma, p \Rightarrow_s p,p, \Delta\)}
\DisplayProof
\mapsto
\AxiomC{}
\RightLabel{Ax}
\UnaryInfC{\(\Gamma, p \Rightarrow_{s} p, \Delta \)}
\DisplayProof
\]

\[
\AxiomC{}
\RightLabel{Ax}
\UnaryInfC{\(\Gamma, q  \Rightarrow_s q, p,p, \Delta\)}
\DisplayProof
\mapsto
\AxiomC{}
\RightLabel{Ax}
\UnaryInfC{\(\Gamma, q \Rightarrow_{s} q, p, \Delta \)}
\DisplayProof
\]
for \(q \neq p\).

\[
\AxiomC{}
\RightLabel{Ax-\(\bot\)}
\UnaryInfC{\(\Gamma,\bot \Rightarrow_s p,p,\Delta\)}
\DisplayProof
\mapsto
\AxiomC{}
\RightLabel{Ax-\(\bot\)}
\UnaryInfC{\(\Gamma,\bot \Rightarrow_{s} p,\Delta\)}
\DisplayProof
\]

\[
\AxiomC{\(\pi_{ 0 }\)}
\noLine
\UnaryInfC{\(\Gamma \Rightarrow_s \Delta,p,p, \phi\)}
\AxiomC{\(\pi_{ 1 }\)}
\noLine
\UnaryInfC{\(\Gamma,\psi \Rightarrow_s \Delta,p,p\)}
\RightLabel{\(\to\)L}
\BinaryInfC{\(\Gamma, \phi \to \psi \Rightarrow_s p,p, \Delta\)}
\DisplayProof
\mapsto
\AxiomC{\(\textsf{rctr}_{ p }(\pi_{ 0})\)}
\noLine
\UnaryInfC{\(\Gamma \Rightarrow_{ s} \Delta,p, \phi\)}
\AxiomC{\(\textsf{rctr}_{ p }(\pi_{ 1 })\)}
\noLine
\UnaryInfC{\(\Gamma, \psi \Rightarrow_{s} \Delta,p\)}
\RightLabel{\(\to\)L}
\BinaryInfC{\(\Gamma, \phi \to \psi \Rightarrow_{s} \Delta,p\)}
\DisplayProof
\]

\[
\AxiomC{\(\pi_{ 0 }\)}
\noLine
\UnaryInfC{\(\Gamma, \phi \Rightarrow_s \psi, \Delta,p,p\)}
\RightLabel{\(\to\)R}
\UnaryInfC{\(\Gamma, \Rightarrow_s \phi \to \psi, \Delta,p,p\)}
\DisplayProof
\mapsto
\AxiomC{\(\textsf{rctr}_{ p }(\pi_{ 0})\)}
\noLine
\UnaryInfC{\(\Gamma, \phi \Rightarrow_{s} \psi, \Delta,p\)}
\RightLabel{\(\to\)R}
\UnaryInfC{\(\Gamma \Rightarrow_{s} \phi \to \psi, \Delta,p\)}
\DisplayProof
\]

\[
\AxiomC{\(\tau \vdash \Gamma, \dnecm \Pi \Rightarrow_\circ \phi\)}
\RightLabel{\(\nec\)}
\UnaryInfC{\(\Sigma, \nec \Gamma, \necm \Pi \Rightarrow_s \nec \phi, \Delta,p,p\)}
\DisplayProof
\mapsto 
\AxiomC{\(\tau \vdash \Gamma, \dnecm \Pi \Rightarrow_\circ \phi\)}
\RightLabel{\(\nec\)}
\UnaryInfC{\(\Sigma, \nec \Gamma, \necm \Pi \Rightarrow_{s} \nec \phi, \Delta,p\)}
\DisplayProof
\]

If \(\pi\) is:
\[
\AxiomC{\(\tau \vdash \Gamma, \dnecm \Pi \Rightarrow_\circ \phi\)}
\AxiomC{\(\pi_{ 0 }\)}
\noLine
\UnaryInfC{\(\Gamma, \dnecm \Pi \Rightarrow_\phi \necm \phi\)}
\RightLabel{\(\necm_f\)}
\BinaryInfC{\(\Sigma,\nec \Gamma, \necm \Pi \Rightarrow_s \necm \phi, \Delta,p,p\)}
\DisplayProof
\]
then it maps to

\[
\AxiomC{\(\tau \vdash \Gamma, \dnecm \Pi \Rightarrow_\circ \phi\)}
\AxiomC{\(\pi_{ 0 }\)}
\noLine
\UnaryInfC{\(\Gamma, \dnecm \Pi \Rightarrow_\phi \necm \phi\)}
\RightLabel{\(\necm_f\)}
\BinaryInfC{\(\Sigma,\nec \Gamma, \necm \Pi \Rightarrow_{s} \necm \phi, \Delta,p\)}
\DisplayProof
\]

If \(\pi\) is:

\[
\AxiomC{\(\tau_0 \vdash \Gamma, \dnecm \Pi \Rightarrow_\circ \phi\)}
\AxiomC{\(\tau_1 \vdash \Gamma, \dnecm \Pi \Rightarrow_\phi \necm \phi\)}
\RightLabel{\(\necm_u\)}
\BinaryInfC{\(\Sigma, \nec \Gamma, \necm \Pi \Rightarrow_s \necm \phi, \Delta,p,p\)}
\DisplayProof
\]

then it maps to

\[
\AxiomC{\(\tau_0 \vdash \Gamma, \dnecm \Pi \Rightarrow_\circ \phi\)}
\AxiomC{\(\tau_1 \vdash \Gamma, \dnecm \Pi \Rightarrow_\phi \necm \phi\)}
\RightLabel{\(\necm_u\)}
\BinaryInfC{\(\Sigma, \nec \Gamma, \necm \Pi \Rightarrow_{s} \necm \phi, \Delta,p\)}
\DisplayProof
\]

\[
\AxiomC{\(\pi_{ 0 }\)}
\noLine
\UnaryInfC{\(\Gamma \Rightarrow_{ s } \Delta,p,p, \chi\)}
\AxiomC{\(\pi_{ 1 }\)}
\noLine
\UnaryInfC{\(\chi, \Gamma \Rightarrow_{ s }\Delta,p,p\)}
\RightLabel{Cut}
\BinaryInfC{\(\Gamma \Rightarrow_{ s }\Delta,p,p\)}
\DisplayProof
\mapsto
\AxiomC{\(\textsf{rctr}_{ p }(\pi_{ 0})\)}
\noLine
\UnaryInfC{\(\Gamma \Rightarrow_{ s } \Delta,p, \chi\)}
\AxiomC{\(\textsf{rctr}_{ p }(\pi_{ 1})\)}
\noLine
\UnaryInfC{\(\chi, \Gamma \Rightarrow_{ s }\Delta,p\)}
\RightLabel{Cut}
\BinaryInfC{\(\Gamma \Rightarrow_{ s }\Delta,p \)}
\DisplayProof
\]
\end{definition}

\begin{lemma}
The function \(\textsf{rctr}_{ p }\) is strongly preserving.
\end{lemma}

\subsection{Contraction of boxed formula in the right}

\begin{definition}
Given \(\pi \vdash \Gamma \Rightarrow_{ s }\Delta,\necm \chi, \necm\chi\) in \(\text{G}^{ \infty }_{ \ell }\text{K}^{ + }_{ s }\) we can define, by recursion on the local height and cases in the last rule apply, the proof \(\textsf{rctr}_{ p }(\pi)\) proof of \(\Gamma \Rightarrow_{ s } \Delta,\necm \chi\) in \(\text{G}^{ \infty }_{ \ell }\text{K}^{ + }_{ s }\) as:

\[
\AxiomC{}
\RightLabel{Ax}
\UnaryInfC{\(\Gamma, p  \Rightarrow_s p, \necm \chi, \necm \chi, \Delta\)}
\DisplayProof
\mapsto
\AxiomC{}
\RightLabel{Ax}
\UnaryInfC{\(\Gamma, p \Rightarrow_{s} p, \necm \chi, \Delta \)}
\DisplayProof
\]

\[
\AxiomC{}
\RightLabel{Ax-\(\bot\)}
\UnaryInfC{\(\Gamma,\bot \Rightarrow_s \necm \chi, \necm \chi,\Delta\)}
\DisplayProof
\mapsto
\AxiomC{}
\RightLabel{Ax-\(\bot\)}
\UnaryInfC{\(\Gamma,\bot \Rightarrow_{s} \necm \chi,\Delta\)}
\DisplayProof
\]

If \(\pi\) is
\[
\AxiomC{\(\pi_{ 0 }\)}
\noLine
\UnaryInfC{\(\Gamma \Rightarrow_s \Delta, \necm \chi,\necm \chi, \phi\)}
\AxiomC{\(\pi_{ 1 }\)}
\noLine
\UnaryInfC{\(\Gamma,\psi \Rightarrow_s \Delta,\necm \chi,\necm \chi\)}
\RightLabel{\(\to\)L}
\BinaryInfC{\(\Gamma, \phi \to \psi \Rightarrow_s \necm \chi,\necm \chi, \Delta\)}
\DisplayProof
\]
then it maps to
\[
\AxiomC{\(\textsf{rctr}_{ \necm \chi }(\pi_{ 0})\)}
\noLine
\UnaryInfC{\(\Gamma \Rightarrow_{ s} \Delta,\necm \chi, \phi\)}
\AxiomC{\(\textsf{rctr}_{ \necm \chi }(\pi_{ 1 })\)}
\noLine
\UnaryInfC{\(\Gamma, \psi \Rightarrow_{s} \Delta,\necm \chi\)}
\RightLabel{\(\to\)L}
\BinaryInfC{\(\Gamma, \phi \to \psi \Rightarrow_{s} \Delta,\necm \chi\)}
\DisplayProof
\]

If \(\pi\) is
\[
\AxiomC{\(\pi_{ 0 }\)}
\noLine
\UnaryInfC{\(\Gamma, \phi \Rightarrow_s \psi, \Delta,\necm \chi,\necm \chi\)}
\RightLabel{\(\to\)R}
\UnaryInfC{\(\Gamma, \Rightarrow_s \phi \to \psi, \Delta,\necm \chi,\necm \chi\)}
\DisplayProof
\mapsto
\AxiomC{\(\textsf{rctr}_{ \necm \chi }(\pi_{ 0})\)}
\noLine
\UnaryInfC{\(\Gamma, \phi \Rightarrow_{s} \psi, \Delta,\necm \chi\)}
\RightLabel{\(\to\)R}
\UnaryInfC{\(\Gamma \Rightarrow_{s} \phi \to \psi, \Delta,\necm \chi\)}
\DisplayProof
\]
then it maps to
\[
\AxiomC{\(\tau \vdash \Gamma, \dnecm \Pi \Rightarrow_\circ \phi\)}
\RightLabel{\(\nec\)}
\UnaryInfC{\(\Sigma, \nec \Gamma, \necm \Pi \Rightarrow_s \nec \phi, \Delta,\necm \chi,\necm \chi\)}
\DisplayProof
\mapsto 
\AxiomC{\(\tau \vdash \Gamma, \dnecm \Pi \Rightarrow_\circ \phi\)}
\RightLabel{\(\nec\)}
\UnaryInfC{\(\Sigma, \nec \Gamma, \necm \Pi \Rightarrow_{s} \nec \phi, \Delta,\necm \chi\)}
\DisplayProof
\]

If \(\pi\) is:
\[
\AxiomC{\(\tau \vdash \Gamma, \dnecm \Pi \Rightarrow_\circ \phi\)}
\AxiomC{\(\pi_{ 0 }\)}
\noLine
\UnaryInfC{\(\Gamma, \dnecm \Pi \Rightarrow_\phi \necm \phi\)}
\RightLabel{\(\necm_f\)}
\BinaryInfC{\(\Sigma,\nec \Gamma, \necm \Pi \Rightarrow_s \necm \phi, \Delta,\necm \chi,\necm \chi\)}
\DisplayProof
\]
then it maps to

\[
\AxiomC{\(\tau \vdash \Gamma, \dnecm \Pi \Rightarrow_\circ \phi\)}
\AxiomC{\(\pi_{ 0 }\)}
\noLine
\UnaryInfC{\(\Gamma, \dnecm \Pi \Rightarrow_\phi \necm \phi\)}
\RightLabel{\(\necm_f\)}
\BinaryInfC{\(\Sigma,\nec \Gamma, \necm \Pi \Rightarrow_{s} \necm \phi, \Delta,\necm \chi\)}
\DisplayProof
\]

If \(\pi\) is:
\[
\AxiomC{\(\tau \vdash \Gamma, \dnecm \Pi \Rightarrow_\circ \chi\)}
\AxiomC{\(\pi_{ 0 }\)}
\noLine
\UnaryInfC{\(\Gamma, \dnecm \Pi \Rightarrow_\chi \necm \chi\)}
\RightLabel{\(\necm_f\)}
\BinaryInfC{\(\Sigma,\nec \Gamma, \necm \Pi \Rightarrow_s \necm \chi, \Delta,\necm \chi\)}
\DisplayProof
\]
then it maps to

\[
\AxiomC{\(\tau \vdash \Gamma, \dnecm \Pi \Rightarrow_\circ \chi\)}
\AxiomC{\(\pi_{ 0 }\)}
\noLine
\UnaryInfC{\(\Gamma, \dnecm \Pi \Rightarrow_\phi \necm \chi\)}
\RightLabel{\(\necm_f\)}
\BinaryInfC{\(\Sigma,\nec \Gamma, \necm \Pi \Rightarrow_{s} \necm \chi, \Delta\)}
\DisplayProof
\]

If \(\pi\) is:

\[
\AxiomC{\(\tau_0 \vdash \Gamma, \dnecm \Pi \Rightarrow_\circ \phi\)}
\AxiomC{\(\tau_1 \vdash \Gamma, \dnecm \Pi \Rightarrow_\phi \necm \phi\)}
\RightLabel{\(\necm_u\)}
\BinaryInfC{\(\Sigma, \nec \Gamma, \necm \Pi \Rightarrow_s \necm \phi, \Delta,\necm \chi,\necm \chi\)}
\DisplayProof
\]

then it maps to

\[
\AxiomC{\(\tau_0 \vdash \Gamma, \dnecm \Pi \Rightarrow_\circ \phi\)}
\AxiomC{\(\tau_1 \vdash \Gamma, \dnecm \Pi \Rightarrow_\phi \necm \phi\)}
\RightLabel{\(\necm_u\)}
\BinaryInfC{\(\Sigma, \nec \Gamma, \necm \Pi \Rightarrow_{s} \necm \phi, \Delta,\necm \chi\)}
\DisplayProof
\]

If \(\pi\) is:

\[
\AxiomC{\(\tau_0 \vdash \Gamma, \dnecm \Pi \Rightarrow_\circ \chi\)}
\AxiomC{\(\tau_1 \vdash \Gamma, \dnecm \Pi \Rightarrow_\chi \necm \chi\)}
\RightLabel{\(\necm_u\)}
\BinaryInfC{\(\Sigma, \nec \Gamma, \necm \Pi \Rightarrow_s \necm \chi, \Delta,\necm \chi\)}
\DisplayProof
\]

then it maps to

\[
\AxiomC{\(\tau_0 \vdash \Gamma, \dnecm \Pi \Rightarrow_\circ \chi\)}
\AxiomC{\(\tau_1 \vdash \Gamma, \dnecm \Pi \Rightarrow_\chi \necm \chi\)}
\RightLabel{\(\necm_u\)}
\BinaryInfC{\(\Sigma, \nec \Gamma, \necm \Pi \Rightarrow_{s} \necm \chi, \Delta\)}
\DisplayProof
\]

\[
\AxiomC{\(\pi_{ 0 }\)}
\noLine
\UnaryInfC{\(\Gamma \Rightarrow_{ s } \Delta,\necm \chi,\necm \chi, \phi\)}
\AxiomC{\(\pi_{ 1 }\)}
\noLine
\UnaryInfC{\(\phi, \Gamma \Rightarrow_{ s }\Delta,\necm \chi,\necm \chi\)}
\RightLabel{Cut}
\BinaryInfC{\(\Gamma \Rightarrow_{ s }\Delta,\necm \chi,\necm \chi\)}
\DisplayProof
\]
maps to
\[
\AxiomC{\(\textsf{rctr}_{ \necm \chi }(\pi_{ 0})\)}
\noLine
\UnaryInfC{\(\Gamma \Rightarrow_{ s } \Delta,\necm \chi, \phi\)}
\AxiomC{\(\textsf{rctr}_{ \necm \chi }(\pi_{ 1})\)}
\noLine
\UnaryInfC{\(\phi, \Gamma \Rightarrow_{ s }\Delta,\necm \chi\)}
\RightLabel{Cut}
\BinaryInfC{\(\Gamma \Rightarrow_{ s }\Delta,\necm \chi \)}
\DisplayProof
\]
\end{definition}

\begin{lemma}
The function \(\textsf{rctr}_{ \necm \chi }\) is strongly preserving.
\end{lemma}

\section{Inversion}
\subsection{Inversion of \(\bot\)}

\begin{definition}
Given \(\pi \vdash \Gamma\Rightarrow_{ s }\Delta,\bot\) in \(\text{G}^{ \infty }_{ \ell }\text{K}^{ + }_{ s }\) we can define, by recursion on the local height and cases in the last rule apply, the proof \(\textsf{inv}_{ \bot }(\pi)\) proof of \(\Gamma \Rightarrow_{ s }\Delta\) in \(\text{G}^{ \infty }_{ \ell }\text{K}^{ + }_{ s }\) as:

\[
\AxiomC{}
\RightLabel{Ax}
\UnaryInfC{\(\Gamma, p \Rightarrow_s p, \Delta,\bot\)}
\DisplayProof
\mapsto
\AxiomC{}
\RightLabel{Ax}
\UnaryInfC{\(\Gamma, p \Rightarrow_{s} p, \Delta \)}
\DisplayProof
\]

\[
\AxiomC{}
\RightLabel{Ax-\(\bot\)}
\UnaryInfC{\(\Gamma,\bot \Rightarrow_s \Delta, \bot\)}
\DisplayProof
\mapsto
\AxiomC{}
\RightLabel{Ax-\(\bot\)}
\UnaryInfC{\(\Gamma,p,\bot \Rightarrow_{s} \Delta\)}
\DisplayProof
\]

\[
\AxiomC{\(\pi_{ 0 }\)}
\noLine
\UnaryInfC{\(\Gamma \Rightarrow_s \Delta,\bot, \phi\)}
\AxiomC{\(\pi_{ 1 }\)}
\noLine
\UnaryInfC{\(\Gamma, \psi \Rightarrow_s \Delta,\bot\)}
\RightLabel{\(\to\)L}
\BinaryInfC{\(\Gamma, \phi \to \psi \Rightarrow_s \Delta, \bot\)}
\DisplayProof
\mapsto
\AxiomC{\(\textsf{inv}_{ \bot }(\pi_{ 0})\)}
\noLine
\UnaryInfC{\(\Gamma \Rightarrow_{ s} \Delta, \phi\)}
\AxiomC{\(\textsf{inv}_{ \bot }(\pi_{ 1 })\)}
\noLine
\UnaryInfC{\(\Gamma, \psi \Rightarrow_{s} \Delta\)}
\RightLabel{\(\to\)L}
\BinaryInfC{\(\Gamma, \phi \to \psi \Rightarrow_{s} \Delta\)}
\DisplayProof
\]

\[
\AxiomC{\(\pi_{ 0 }\)}
\noLine
\UnaryInfC{\(\Gamma, \phi \Rightarrow_s \psi, \Delta, \bot\)}
\RightLabel{\(\to\)R}
\UnaryInfC{\(\Gamma \Rightarrow_s \phi \to \psi, \Delta, \bot\)}
\DisplayProof
\mapsto
\AxiomC{\(\textsf{inv}_{ \bot }(\pi_{ 0})\)}
\noLine
\UnaryInfC{\(\Gamma, \phi \Rightarrow_{s} \psi, \Delta\)}
\RightLabel{\(\to\)R}
\UnaryInfC{\(\Gamma \Rightarrow_{s} \phi \to \psi, \Delta\)}
\DisplayProof
\]

\[
\AxiomC{\(\tau \vdash \Gamma, \dnecm \Pi \Rightarrow_\circ \phi\)}
\RightLabel{\(\nec\)}
\UnaryInfC{\(\Sigma, \nec \Gamma, \necm \Pi \Rightarrow_s \nec \phi, \Delta,\bot\)}
\DisplayProof
\mapsto 
\AxiomC{\(\tau \vdash \Gamma, \dnecm \Pi \Rightarrow_\circ \phi\)}
\RightLabel{\(\nec\)}
\UnaryInfC{\(\Sigma, \nec \Gamma, \necm \Pi \Rightarrow_{s} \nec \phi, \Delta\)}
\DisplayProof
\]

If \(\pi\) is:
\[
\AxiomC{\(\tau \vdash \Gamma, \dnecm \Pi \Rightarrow_\circ \phi\)}
\AxiomC{\(\pi_{ 0 }\)}
\noLine
\UnaryInfC{\(\Gamma, \dnecm \Pi \Rightarrow_\phi \necm \phi\)}
\RightLabel{\(\necm_f\)}
\BinaryInfC{\(\Sigma,\nec \Gamma, \necm \Pi \Rightarrow_s \necm \phi, \Delta,\bot\)}
\DisplayProof
\]
then it maps to
\[
\AxiomC{\(\tau \vdash \Gamma, \dnecm \Pi \Rightarrow_\circ \phi\)}
\AxiomC{\(\pi_{ 0 }\)}
\noLine
\UnaryInfC{\(\Gamma, \dnecm \Pi \Rightarrow_\phi \necm \phi\)}
\RightLabel{\(\necm_f\)}
\BinaryInfC{\(\Sigma,\nec \Gamma, \necm \Pi \Rightarrow_{s} \necm \phi, \Delta\)}
\DisplayProof
\]

If \(\pi\) is:
\[
\AxiomC{\(\tau_0 \vdash \Gamma, \dnecm \Pi \Rightarrow_\circ \phi\)}
\AxiomC{\(\tau_1 \vdash \Gamma, \dnecm \Pi \Rightarrow_\phi \necm \phi\)}
\RightLabel{\(\necm_u\)}
\BinaryInfC{\(\Sigma, \nec \Gamma, \necm \Pi \Rightarrow_s \necm \phi, \Delta, \bot\)}
\DisplayProof
\]
then it maps to
\[
\AxiomC{\(\tau_0 \vdash \Gamma, \dnecm \Pi \Rightarrow_\circ \phi\)}
\AxiomC{\(\tau_1 \vdash \Gamma, \dnecm \Pi \Rightarrow_\phi \necm \phi\)}
\RightLabel{\(\necm_u\)}
\BinaryInfC{\(\Sigma, \nec \Gamma, \necm \Pi \Rightarrow_{s} \necm \phi, \Delta\)}
\DisplayProof
\]

\[
\AxiomC{\(\pi_{ 0 }\)}
\noLine
\UnaryInfC{\(\Gamma \Rightarrow_{ s } \Delta,\bot, \chi\)}
\AxiomC{\(\pi_{ 1 }\)}
\noLine
\UnaryInfC{\(\chi, \Gamma \Rightarrow_{ s }\Delta, \bot\)}
\RightLabel{Cut}
\BinaryInfC{\(\Gamma \Rightarrow_{ s }\Delta, \bot\)}
\DisplayProof
\mapsto
\AxiomC{\(\textsf{inv}_{ \bot }(\pi_{ 0})\)}
\noLine
\UnaryInfC{\(\Gamma \Rightarrow_{ s } \Delta, \chi\)}
\AxiomC{\(\textsf{inv}_{ \bot }(\pi_{ 1})\)}
\noLine
\UnaryInfC{\(\chi, \Gamma \Rightarrow_{ s }\Delta \)}
\RightLabel{Cut}
\BinaryInfC{\(\Gamma, \Rightarrow_{ s }\Delta \)}
\DisplayProof
\]
\end{definition}

\begin{lemma}
The function \(\textsf{inv}_{ \bot }\) is strongly preserving.
\end{lemma}

\subsection{Inversion of \(\to\text{L}\)}
\subsubsection{Left inversion}

\begin{definition}
Given \(\pi \vdash \Gamma, \chi_{ 0 }\to \chi_{ 1 }\Rightarrow_{ s }\Delta\) in \(\text{G}^{ \infty }_{ \ell }\text{K}^{ + }_{ s }\) we can define, by recursion on the local height and cases in the last rule apply, the proof \(\textsf{linv}_{ \chi_{ 0 } \to \chi_{ 1 }}^{ 0 }(\pi)\) proof of \(\Gamma \Rightarrow_{ s }\Delta,\chi_{ 0 }\) in \(\text{G}^{ \infty }_{ \ell }\text{K}^{ + }_{ s }\) as:

\[
\AxiomC{}
\RightLabel{Ax}
\UnaryInfC{\(\Gamma, \chi_{ 0 }\to \chi_{ 1 }, p \Rightarrow_s p, \Delta\)}
\DisplayProof
\mapsto
\AxiomC{}
\RightLabel{Ax}
\UnaryInfC{\(\Gamma, p \Rightarrow_{s} p, \Delta, \chi_{ 0 }\)}
\DisplayProof
\]

\[
\AxiomC{}
\RightLabel{Ax-\(\bot\)}
\UnaryInfC{\(\Gamma,\chi_{ 0 }\to \chi_{ 1 },\bot \Rightarrow_s \Delta\)}
\DisplayProof
\mapsto
\AxiomC{}
\RightLabel{Ax-\(\bot\)}
\UnaryInfC{\(\Gamma,\bot \Rightarrow_{s} \Delta, \chi_{ 0 }\)}
\DisplayProof
\]

If \(\pi\) is of shape:
\[
\AxiomC{\(\pi_{ 0 }\)}
\noLine
\UnaryInfC{\(\Gamma, \chi_{ 0 }\to \chi_{ 1 } \Rightarrow_s \Delta, \phi\)}
\AxiomC{\(\pi_{ 1 }\)}
\noLine
\UnaryInfC{\(\Gamma, \chi_{ 0 }\to \chi_{ 1 }, \psi \Rightarrow_s \Delta\)}
\RightLabel{\(\to\)L}
\BinaryInfC{\(\Gamma, \chi_{ 0 }\to \chi_{ 1 }, \phi \to \psi \Rightarrow_s \Delta\)}
\DisplayProof
\]
we transform it to
\[
\AxiomC{\(\textsf{linv}_{ \chi_{ 0 } \to \chi_{ 1 }}^{ 0 }(\pi_{ 0})\)}
\noLine
\UnaryInfC{\(\Gamma \Rightarrow_{ s} \Delta, \chi_{ 0 }, \phi\)}
\AxiomC{\(\textsf{linv}_{ \chi_{ 0 } \to \chi_{ 1 }}^{ 0 }(\pi_{ 1 })\)}
\noLine
\UnaryInfC{\(\Gamma, \psi \Rightarrow_{s} \Delta, \chi_{ 0 }\)}
\RightLabel{\(\to\)L}
\BinaryInfC{\(\Gamma, \phi \to \psi \Rightarrow_{s} \Delta, \chi_{ 0 }\)}
\DisplayProof
\]

If \(\pi\) is of shape:
\[
\AxiomC{\(\pi_{ 0 }\)}
\noLine
\UnaryInfC{\(\Gamma \Rightarrow_s \Delta, \chi_{ 0 }\)}
\AxiomC{\(\pi_{ 1 }\)}
\noLine
\UnaryInfC{\(\Gamma, \chi_{ 1 } \Rightarrow_s \Delta\)}
\RightLabel{\(\to\)L}
\BinaryInfC{\(\Gamma, \chi_{ 0 }\to \chi_{ 1 }\Rightarrow_s \Delta\)}
\DisplayProof
\]
the desired proof is \(\pi_{ 0 }\).

\[
\AxiomC{\(\pi_{ 0 }\)}
\noLine
\UnaryInfC{\(\Gamma, \chi_{ 0 }\to \chi_{ 1 }, \phi \Rightarrow_s \psi, \Delta\)}
\RightLabel{\(\to\)R}
\UnaryInfC{\(\Gamma, \chi_{ 0 }\to \chi_{ 1 } \Rightarrow_s \phi \to \psi, \Delta\)}
\DisplayProof
\mapsto
\AxiomC{\(\textsf{linv}_{ \chi_{ 0 }\to\chi_{ 1 }}^{ 0 }(\pi_{ 0})\)}
\noLine
\UnaryInfC{\(\Gamma, \phi \Rightarrow_{s} \psi, \Delta, \chi_{ 0 }\)}
\RightLabel{\(\to\)R}
\UnaryInfC{\(\Gamma \Rightarrow_{s} \phi \to \psi, \Delta,\chi_{ 0 }\)}
\DisplayProof
\]

\[
\AxiomC{\(\tau \vdash \Gamma, \dnecm \Pi \Rightarrow_\circ \phi\)}
\RightLabel{\(\nec\)}
\UnaryInfC{\(\Sigma, \chi_{ 0 }\to \chi_{ 1 }, \nec \Gamma, \necm \Pi \Rightarrow_s \nec \phi, \Delta\)}
\DisplayProof
\mapsto 
\AxiomC{\(\tau \vdash \Gamma, \dnecm \Pi \Rightarrow_\circ \phi\)}
\RightLabel{\(\nec\)}
\UnaryInfC{\(\Sigma, \nec \Gamma, \necm \Pi \Rightarrow_{s} \nec \phi, \Delta, \chi_{ 0 }\)}
\DisplayProof
\]

If \(\pi\) is:
\[
\AxiomC{\(\tau \vdash \Gamma, \dnecm \Pi \Rightarrow_\circ \phi\)}
\AxiomC{\(\pi_{ 0 }\)}
\noLine
\UnaryInfC{\(\Gamma, \dnecm \Pi \Rightarrow_\phi \necm \phi\)}
\RightLabel{\(\necm_f\)}
\BinaryInfC{\(\Sigma,\chi_{ 0 }\to \chi_{ 1 },\nec \Gamma, \necm \Pi \Rightarrow_s \necm \phi, \Delta\)}
\DisplayProof
\]
then it maps to
\[
\AxiomC{\(\tau \vdash \Gamma, \dnecm \Pi \Rightarrow_\circ \phi\)}
\AxiomC{\(\pi_{ 0 }\)}
\noLine
\UnaryInfC{\(\Gamma, \dnecm \Pi \Rightarrow_\phi \necm \phi\)}
\RightLabel{\(\necm_f\)}
\BinaryInfC{\(\Sigma,\nec \Gamma, \necm \Pi \Rightarrow_{s} \necm \phi, \Delta, \chi_{ 0 }\)}
\DisplayProof
\]

If \(\pi\) is
\[
\AxiomC{\(\tau_0 \vdash \Gamma, \dnecm \Pi \Rightarrow_\circ \phi\)}
\AxiomC{\(\tau_1 \vdash \Gamma, \dnecm \Pi \Rightarrow_\phi \necm \phi\)}
\RightLabel{\(\necm_u\)}
\BinaryInfC{\(\Sigma, \chi_{ 0 }\to \chi_{ 1 }, \nec \Gamma, \necm \Pi \Rightarrow_s \necm \phi, \Delta\)}
\DisplayProof
\]
then it maps to
\[
\AxiomC{\(\tau_0 \vdash \Gamma, \dnecm \Pi \Rightarrow_\circ \phi\)}
\AxiomC{\(\tau_1 \vdash \Gamma, \dnecm \Pi \Rightarrow_\phi \necm \phi\)}
\RightLabel{\(\necm_u\)}
\BinaryInfC{\(\Sigma, \nec \Gamma, \necm \Pi \Rightarrow_{s} \necm \phi, \Delta, \chi_{ 0 }\)}
\DisplayProof
\]

If \(\pi\) is
\[
\AxiomC{\(\pi_{ 0 }\)}
\noLine
\UnaryInfC{\(\Gamma, \chi_{ 0 }\to \chi_{ 1 } \Rightarrow_{ s } \Delta, \phi\)}
\AxiomC{\(\pi_{ 1 }\)}
\noLine
\UnaryInfC{\(\phi, \Gamma, \chi_{ 0 }\to \chi_{ 1 } \Rightarrow_{ s }\Delta\)}
\RightLabel{Cut}
\BinaryInfC{\(\Gamma, \chi_{ 0 }\to \chi_{ 1 } \Rightarrow_{ s }\Delta\)}
\DisplayProof
\]
then it maps to
\[
\AxiomC{\(\textsf{linv}_{ \chi_{ 0 } \to \chi_{ 1 }}^{ 0 }(\pi_{ 0})\)}
\noLine
\UnaryInfC{\(\Gamma \Rightarrow_{ s } \Delta, \chi_{ 0 }, \phi\)}
\AxiomC{\(\textsf{linv}_{ \chi_{ 0 } \to \chi_{ 1 }}^{ 0 }(\pi_{ 1})\)}
\noLine
\UnaryInfC{\(\phi, \Gamma \Rightarrow_{ s }\Delta, \chi_{ 0 } \)}
\RightLabel{Cut}
\BinaryInfC{\(\Gamma, \Rightarrow_{ s }\Delta, \chi_{ 0 } \)}
\DisplayProof
\]
\end{definition}

\begin{lemma}
The function \(\textsf{linv}_{ \chi_{ 0 }\to\chi_{ 1 }}^{ 0 }\) is strongly preserving.
\end{lemma}
\subsubsection{Right inversion}

\begin{definition}
Given \(\pi \vdash \Gamma, \chi_{ 0 }\to \chi_{ 1 }\Rightarrow_{ s }\Delta\) in \(\text{G}^{ \infty }_{ \ell }\text{K}^{ + }_{ s }\) we can define, by recursion on the local height and cases in the last rule apply, the proof \(\textsf{linv}_{ \chi_{ 0 } \to \chi_{ 1 }}^{ 1 }(\pi)\) proof of \(\Gamma, \chi_{ 1 } \Rightarrow_{ s }\Delta_{}\) in \(\text{G}^{ \infty }_{ \ell }\text{K}^{ + }_{ s }\) as:

\[
\AxiomC{}
\RightLabel{Ax}
\UnaryInfC{\(\Gamma, \chi_{ 0 }\to \chi_{ 1 }, p \Rightarrow_s p, \Delta\)}
\DisplayProof
\mapsto
\AxiomC{}
\RightLabel{Ax}
\UnaryInfC{\(\Gamma, \chi_{ 1 }, p \Rightarrow_{s} p, \Delta\)}
\DisplayProof
\]

\[
\AxiomC{}
\RightLabel{Ax-\(\bot\)}
\UnaryInfC{\(\Gamma,\chi_{ 0 }\to \chi_{ 1 },\bot \Rightarrow_s \Delta\)}
\DisplayProof
\mapsto
\AxiomC{}
\RightLabel{Ax-\(\bot\)}
\UnaryInfC{\(\Gamma,\chi_{ 1 },\bot \Rightarrow_{s} \Delta\)}
\DisplayProof
\]

If \(\pi\) is of shape:
\[
\AxiomC{\(\pi_{ 0 }\)}
\noLine
\UnaryInfC{\(\Gamma, \chi_{ 0 }\to \chi_{ 1 } \Rightarrow_s \Delta, \phi\)}
\AxiomC{\(\pi_{ 1 }\)}
\noLine
\UnaryInfC{\(\Gamma, \chi_{ 0 }\to \chi_{ 1 }, \psi \Rightarrow_s \Delta\)}
\RightLabel{\(\to\)L}
\BinaryInfC{\(\Gamma, \chi_{ 0 }\to \chi_{ 1 }, \phi \to \psi \Rightarrow_s \Delta\)}
\DisplayProof
\]
we transform it to
\[
\AxiomC{\(\textsf{linv}_{ \chi_{ 0 } \to \chi_{ 1 }}^{ 1 }(\pi_{ 0})\)}
\noLine
\UnaryInfC{\(\Gamma, \chi_{ 1 } \Rightarrow_{ s} \Delta, \phi\)}
\AxiomC{\(\textsf{linv}_{ \chi_{ 0 } \to \chi_{ 1 }}^{ 1 }(\pi_{ 1 })\)}
\noLine
\UnaryInfC{\(\Gamma, \chi_{ 1 }, \psi \Rightarrow_{s} \Delta\)}
\RightLabel{\(\to\)L}
\BinaryInfC{\(\Gamma, \chi_{ 1 }, \phi \to \psi \Rightarrow_{s} \Delta\)}
\DisplayProof
\]

If \(\pi\) is of shape:
\[
\AxiomC{\(\pi_{ 0 }\)}
\noLine
\UnaryInfC{\(\Gamma \Rightarrow_s \Delta, \chi_{ 0 }\)}
\AxiomC{\(\pi_{ 1 }\)}
\noLine
\UnaryInfC{\(\Gamma, \chi_{ 1 } \Rightarrow_s \Delta\)}
\RightLabel{\(\to\)L}
\BinaryInfC{\(\Gamma, \chi_{ 0 }\to \chi_{ 1 }\Rightarrow_s \Delta\)}
\DisplayProof
\]
the desired proof is \(\pi_{ 1 }\).

\[
\AxiomC{\(\pi_{ 0 }\)}
\noLine
\UnaryInfC{\(\Gamma, \chi_{ 0 }\to \chi_{ 1 }, \phi \Rightarrow_s \psi, \Delta\)}
\RightLabel{\(\to\)R}
\UnaryInfC{\(\Gamma, \chi_{ 0 }\to \chi_{ 1 } \Rightarrow_s \phi \to \psi, \Delta\)}
\DisplayProof
\mapsto
\AxiomC{\(\textsf{linv}_{ \chi_{ 0 }\to\chi_{ 1 }}^{ 1 }(\pi_{ 0})\)}
\noLine
\UnaryInfC{\(\Gamma, \chi_{ 1 }, \phi \Rightarrow_{s} \psi, \Delta\)}
\RightLabel{\(\to\)R}
\UnaryInfC{\(\Gamma, \chi_{ 1 } \Rightarrow_{s} \phi \to \psi, \Delta\)}
\DisplayProof
\]

\[
\AxiomC{\(\tau \vdash \Gamma, \dnecm \Pi \Rightarrow_\circ \phi\)}
\RightLabel{\(\nec\)}
\UnaryInfC{\(\Sigma, \chi_{ 0 }\to \chi_{ 1 }, \nec \Gamma, \necm \Pi \Rightarrow_s \nec \phi, \Delta\)}
\DisplayProof
\mapsto 
\AxiomC{\(\tau \vdash \Gamma, \dnecm \Pi \Rightarrow_\circ \phi\)}
\RightLabel{\(\nec\)}
\UnaryInfC{\(\Sigma, \chi_{ 1 }, \nec \Gamma, \necm \Pi \Rightarrow_{s} \nec \phi, \Delta\)}
\DisplayProof
\]

If \(\pi\) is:
\[
\AxiomC{\(\tau \vdash \Gamma, \dnecm \Pi \Rightarrow_\circ \phi\)}
\AxiomC{\(\pi_{ 0 }\)}
\noLine
\UnaryInfC{\(\Gamma, \dnecm \Pi \Rightarrow_\phi \necm \phi\)}
\RightLabel{\(\necm_f\)}
\BinaryInfC{\(\Sigma,\chi_{ 0 }\to \chi_{ 1 },\nec \Gamma, \necm \Pi \Rightarrow_s \necm \phi, \Delta\)}
\DisplayProof
\]
then it maps to
\[
\AxiomC{\(\tau \vdash \Gamma, \dnecm \Pi \Rightarrow_\circ \phi\)}
\AxiomC{\(\pi_{ 0 }\)}
\noLine
\UnaryInfC{\(\Gamma, \dnecm \Pi \Rightarrow_\phi \necm \phi\)}
\RightLabel{\(\necm_f\)}
\BinaryInfC{\(\Sigma, \chi_{ 1 },\nec \Gamma, \necm \Pi \Rightarrow_{s} \necm \phi, \Delta\)}
\DisplayProof
\]

If \(\pi\) is
\[
\AxiomC{\(\tau_0 \vdash \Gamma, \dnecm \Pi \Rightarrow_\circ \phi\)}
\AxiomC{\(\tau_1 \vdash \Gamma, \dnecm \Pi \Rightarrow_\phi \necm \phi\)}
\RightLabel{\(\necm_u\)}
\BinaryInfC{\(\Sigma, \chi_{ 0 }\to \chi_{ 1 }, \nec \Gamma, \necm \Pi \Rightarrow_s \necm \phi, \Delta\)}
\DisplayProof
\]
then it maps to
\[
\AxiomC{\(\tau_0 \vdash \Gamma, \dnecm \Pi \Rightarrow_\circ \phi\)}
\AxiomC{\(\tau_1 \vdash \Gamma, \dnecm \Pi \Rightarrow_\phi \necm \phi\)}
\RightLabel{\(\necm_u\)}
\BinaryInfC{\(\Sigma, \chi_{ 1 }, \nec \Gamma, \necm \Pi \Rightarrow_{s} \necm \phi, \Delta\)}
\DisplayProof
\]

If \(\pi\) is
\[
\AxiomC{\(\pi_{ 0 }\)}
\noLine
\UnaryInfC{\(\Gamma, \chi_{ 0 }\to \chi_{ 1 } \Rightarrow_{ s } \Delta, \phi\)}
\AxiomC{\(\pi_{ 1 }\)}
\noLine
\UnaryInfC{\(\phi, \Gamma, \chi_{ 0 }\to \chi_{ 1 } \Rightarrow_{ s }\Delta\)}
\RightLabel{Cut}
\BinaryInfC{\(\Gamma, \chi_{ 0 }\to \chi_{ 1 } \Rightarrow_{ s }\Delta\)}
\DisplayProof
\]
then it maps to
\[
\AxiomC{\(\textsf{linv}_{ \chi_{ 0 } \to \chi_{ 1 }}^{ 1 }(\pi_{ 0})\)}
\noLine
\UnaryInfC{\(\Gamma, \chi_{ 1 } \Rightarrow_{ s } \Delta, \phi\)}
\AxiomC{\(\textsf{linv}_{ \chi_{ 0 } \to \chi_{ 1 }}^{ 1 }(\pi_{ 1})\)}
\noLine
\UnaryInfC{\(\phi, \Gamma, \chi_{ 1 } \Rightarrow_{ s }\Delta\)}
\RightLabel{Cut}
\BinaryInfC{\(\Gamma, \chi_{ 1 } \Rightarrow_{ s }\Delta\)}
\DisplayProof
\]
\end{definition}

\begin{lemma}
The function \(\textsf{linv}_{ \chi_{ 0 }\to\chi_{1}}^{1}\) is strongly preserving.
\end{lemma}
\subsection{Inversion of \(\to\text{R}\)}
\begin{definition}
Given \(\pi \vdash \Gamma \Rightarrow_{ s} \chi_{ 0 }\to \chi_{ 1 }, \Delta\) in \(\text{G}^{ \infty }_{ \ell }\text{K}^{ + }_{ s }\) we can define, by recursion on the local height and cases in the last rule apply, the proof \(\textsf{rinv}_{ \chi_{ 0 } \to \chi_{ 1 }}(\pi)\) proof of \(\Gamma, \chi_{ 0 } \Rightarrow_{ s }\Delta,\chi_{ 1 }\) in \(\text{G}^{ \infty }_{ \ell }\text{K}^{ + }_{ s }\) as:

\[
\AxiomC{}
\RightLabel{Ax}
\UnaryInfC{\(\Gamma, p \Rightarrow_s p, \Delta, \chi_{ 0 }\to \chi_{ 1 }\)}
\DisplayProof
\mapsto
\AxiomC{}
\RightLabel{Ax}
\UnaryInfC{\(\Gamma, \chi_{ 0 }, p \Rightarrow_{s} p, \Delta, \chi_{ 1 }\)}
\DisplayProof
\]

\[
\AxiomC{}
\RightLabel{Ax-\(\bot\)}
\UnaryInfC{\(\Gamma,\bot \Rightarrow_s \Delta, \chi_{ 0 }\to \chi_{ 1 }\)}
\DisplayProof
\mapsto
\AxiomC{}
\RightLabel{Ax-\(\bot\)}
\UnaryInfC{\(\Gamma,\chi_{ 0 },\bot \Rightarrow_{s} \Delta, \chi_{ 1 }\)}
\DisplayProof
\]

If \(\pi\) is of shape:
\[
\AxiomC{\(\pi_{ 0 }\)}
\noLine
\UnaryInfC{\(\Gamma \Rightarrow_s \Delta, \chi_{ 0 }\to \chi_{ 1 }, \phi\)}
\AxiomC{\(\pi_{ 1 }\)}
\noLine
\UnaryInfC{\(\Gamma, \psi \Rightarrow_s \Delta, \chi_{ 0 }\to \chi_{ 1 }\)}
\RightLabel{\(\to\)L}
\BinaryInfC{\(\Gamma, \phi \to \psi \Rightarrow_s \Delta, , \chi_{ 0 }\to \chi_{ 1 }\)}
\DisplayProof
\]
we transform it to
\[
\AxiomC{\(\textsf{rinv}_{ \chi_{ 0 } \to \chi_{ 1 }}(\pi_{ 0})\)}
\noLine
\UnaryInfC{\(\Gamma, \chi_{ 0 } \Rightarrow_{ s} \Delta, \chi_{ 1 }, \phi\)}
\AxiomC{\(\textsf{rinv}_{ \chi_{ 0 } \to \chi_{ 1 }}(\pi_{ 1 })\)}
\noLine
\UnaryInfC{\(\Gamma, \chi_{ 0 }, \psi \Rightarrow_{s} \Delta, \chi_{ 1 }\)}
\RightLabel{\(\to\)L}
\BinaryInfC{\(\Gamma, \chi_{ 0 }, \phi \to \psi \Rightarrow_{s} \Delta, \chi_{ 1 }\)}
\DisplayProof
\]

\[
\AxiomC{\(\pi_{ 0 }\)}
\noLine
\UnaryInfC{\(\Gamma, \phi \Rightarrow_s \psi, \Delta, \chi_{ 0 }\to \chi_{ 1}\)}
\RightLabel{\(\to\)R}
\UnaryInfC{\(\Gamma \Rightarrow_s \phi \to \psi, \Delta, \chi_{ 0 }\to \chi_{ 1}\)}
\DisplayProof
\mapsto
\AxiomC{\(\textsf{rinv}_{ \chi_{ 0 }\to\chi_{ 1 }}(\pi_{ 0})\)}
\noLine
\UnaryInfC{\(\Gamma, \chi_{ 0 }, \phi \Rightarrow_{s} \psi, \Delta, \chi_{ 1 }\)}
\RightLabel{\(\to\)R}
\UnaryInfC{\(\Gamma, \chi_{ 0 } \Rightarrow_{s} \phi \to \psi, \Delta, \chi_{ 1 }\)}
\DisplayProof
\]

\[
\AxiomC{\(\pi_{ 0 }\)}
\noLine
\UnaryInfC{\(\Gamma, \chi_{ 0 } \Rightarrow_s \chi_{ 1 }, \Delta\)}
\RightLabel{\(\to\)R}
\UnaryInfC{\(\Gamma \Rightarrow_s \chi_{ 0 } \to \chi_{ 1 }, \Delta\)}
\DisplayProof
\]
the desired proof is \(\pi_{ 0 }\).

\[
\AxiomC{\(\tau \vdash \Gamma, \dnecm \Pi \Rightarrow_\circ \phi\)}
\RightLabel{\(\nec\)}
\UnaryInfC{\(\Sigma, \nec \Gamma, \necm \Pi \Rightarrow_s \nec \phi, \Delta, \chi_{ 0 }\to \chi_{ 1 }\)}
\DisplayProof
\mapsto 
\AxiomC{\(\tau \vdash \Gamma, \dnecm \Pi \Rightarrow_\circ \phi\)}
\RightLabel{\(\nec\)}
\UnaryInfC{\(\Sigma, \chi_{ 0 }, \nec \Gamma, \necm \Pi \Rightarrow_{s} \nec \phi, \Delta, \chi_{ 1 }\)}
\DisplayProof
\]

If \(\pi\) is:
\[
\AxiomC{\(\tau \vdash \Gamma, \dnecm \Pi \Rightarrow_\circ \phi\)}
\AxiomC{\(\pi_{ 0 }\)}
\noLine
\UnaryInfC{\(\Gamma, \dnecm \Pi \Rightarrow_\phi \necm \phi\)}
\RightLabel{\(\necm_f\)}
\BinaryInfC{\(\Sigma,\nec \Gamma, \necm \Pi \Rightarrow_s \necm \phi, \Delta, \chi_{ 0 }\to \chi_{ 1 }\)}
\DisplayProof
\]
then it maps to
\[
\AxiomC{\(\tau \vdash \Gamma, \dnecm \Pi \Rightarrow_\circ \phi\)}
\AxiomC{\(\pi_{ 0 }\)}
\noLine
\UnaryInfC{\(\Gamma, \dnecm \Pi \Rightarrow_\phi \necm \phi\)}
\RightLabel{\(\necm_f\)}
\BinaryInfC{\(\Sigma, \chi_{ 0 },\nec \Gamma, \necm \Pi \Rightarrow_{s} \necm \phi, \Delta, \chi_{ 1 }\)}
\DisplayProof
\]

If \(\pi\) is
\[
\AxiomC{\(\tau_0 \vdash \Gamma, \dnecm \Pi \Rightarrow_\circ \phi\)}
\AxiomC{\(\tau_1 \vdash \Gamma, \dnecm \Pi \Rightarrow_\phi \necm \phi\)}
\RightLabel{\(\necm_u\)}
\BinaryInfC{\(\Sigma, \nec \Gamma, \necm \Pi \Rightarrow_s \necm \phi, \Delta,  \chi_{ 0 }\to \chi_{ 1 }\)}
\DisplayProof
\]
then it maps to
\[
\AxiomC{\(\tau_0 \vdash \Gamma, \dnecm \Pi \Rightarrow_\circ \phi\)}
\AxiomC{\(\tau_1 \vdash \Gamma, \dnecm \Pi \Rightarrow_\phi \necm \phi\)}
\RightLabel{\(\necm_u\)}
\BinaryInfC{\(\Sigma, \chi_{ 0 }, \nec \Gamma, \necm \Pi \Rightarrow_{s} \necm \phi, \Delta, \chi_{ 1 }\)}
\DisplayProof
\]

If \(\pi\) is
\[
\AxiomC{\(\pi_{ 0 }\)}
\noLine
\UnaryInfC{\(\Gamma \Rightarrow_{ s } \Delta, , \chi_{ 0 }\to \chi_{ 1 }, \phi\)}
\AxiomC{\(\pi_{ 1 }\)}
\noLine
\UnaryInfC{\(\phi, \Gamma \Rightarrow_{ s }\Delta, \chi_{ 0 }\to \chi_{ 1 }\)}
\RightLabel{Cut}
\BinaryInfC{\(\Gamma \Rightarrow_{ s }\Delta, \chi_{ 0 }\to \chi_{ 1 }\)}
\DisplayProof
\]
then it maps to
\[
\AxiomC{\(\textsf{rinv}_{ \chi_{ 0 } \to \chi_{ 1 }}(\pi_{ 0})\)}
\noLine
\UnaryInfC{\(\Gamma, \chi_{ 0 } \Rightarrow_{ s } \Delta, \chi_{ 1 }, \phi\)}
\AxiomC{\(\textsf{rinv}_{ \chi_{ 0 } \to \chi_{ 1 }}(\pi_{ 1})\)}
\noLine
\UnaryInfC{\(\phi, \Gamma, \chi_{ 0 } \Rightarrow_{ s }\Delta, \chi_{ 1 }\)}
\RightLabel{Cut}
\BinaryInfC{\(\Gamma, \chi_{ 0 } \Rightarrow_{ s }\Delta, \chi_{ 1 }\)}
\DisplayProof
\]
\end{definition}

\begin{lemma}
The function \(\textsf{rinv}_{ \chi_0\to\chi_1}\) is strongly preserving.
\end{lemma}

\newpage
\section{Cut reductions}\label{sec:cut-reductions}

Here we display all the cut-reductions we are going to use.
We will write \(\textsf{cut}_i\) with \(i \in \NN\) to have a way to refer to each instance of cut since when we use the cut-reductions in proofs, it may be the case that we leave cuts as cuts or that we replace them by some proof provided by cut-admissibility.

We notice that the cut reductions of this section provides all possible cut reductions when we have a \(\pi \vdash \Gamma \Rightarrow_s \Delta, \chi\) and \(\tau \vdash \chi, \Gamma \Rightarrow_s \Delta\) for an arbitrary \(\chi\) (the cut formula).
In case \(\chi\) is of a particular shape (e.g.\ a propositional variable or a box formula) we will mention explicitly in the proof which are the cut reductions that appear.
We are prepared to write all the cut reductions:

\subsection{Labeling and cut formula coincides}\label{sec:cut-reduction-label}

Let \(\chi = \necm \chi_0\) and \(s = \chi_0\). Since \(\chi, \Gamma \Rightarrow_s \Delta\) is the conclusion of \(\tau\) and \(s\) is a formula, we know that \(\necm s = \necm \chi_0\) must appear in \(\Delta\).\footnote{We remember that, by definition of annotated sequent, if \(\Gamma \Rightarrow_\phi \Delta\) then \(\necm \phi\) must occur in \(\Delta\).}  
Then \(\Delta = \Delta_0, \chi\) and the conclusion of \(\pi\) would be \(\Gamma \Rightarrow_s \Delta_0, \necm \chi_0, \chi\), i.e.\ \(\Gamma \Rightarrow_s \Delta_0, \chi, \chi\).

We have the cut reduction
\[
	\AxiomC{\(\pi\)}
	\noLine
	\UnaryInfC{\(\Gamma \Rightarrow_s \Delta_0, \chi, \chi\)}
	\DisplayProof
	\AxiomC{\(\tau\)}
	\noLine
	\UnaryInfC{\(\chi, \Gamma \Rightarrow_s \Delta_0, \chi\)}
	\DisplayProof
	\mapsto
	\AxiomC{\(\textsf{rctr}_{\necm \chi_0}(\pi)\)}
	\noLine
	\UnaryInfC{\(\Gamma \Rightarrow_s \Delta_0, \chi\)}
	\DisplayProof
\]

From now own we assume that either \(\chi\) is not of shape \(\necm \chi_0\) or if it is then \(s \neq \chi_0\).

\subsection{Axiomatic cases}
Assume that either \(\pi\) or \(\tau\) is axiomatic.

\textbf{Axiomatic character is in side formula}
\[
	\AxiomC{}
	\UnaryInfC{\(\Gamma, p \Rightarrow_s p, \Delta, \chi\)}
	\DisplayProof
	\hspace{0.3cm}
	\AxiomC{\(\tau\)}
	\noLine
	\UnaryInfC{\(\chi, \Gamma, p \Rightarrow_s p, \Delta\)}
	\DisplayProof
	\mapsto
	\AxiomC{}
	\UnaryInfC{\(\Gamma, p \Rightarrow_s p, \Delta\)}
	\DisplayProof
\]
\[
	\AxiomC{\(\pi\)}
	\noLine
	\UnaryInfC{\(\Gamma, p \Rightarrow_s p, \Delta, \chi\)}
	\DisplayProof
	\hspace{0.3cm}
	\AxiomC{}
	\UnaryInfC{\(\chi, \Gamma, p \Rightarrow_s p, \Delta\)}
	\DisplayProof
	\mapsto
	\AxiomC{}
	\UnaryInfC{\(\Gamma, p \Rightarrow_s p, \Delta\)}
	\DisplayProof
\]
\[
	\AxiomC{}
	\UnaryInfC{\(\Gamma, \bot \Rightarrow_s \Delta, \chi\)}
	\DisplayProof
	\hspace{0.3cm}
	\AxiomC{\(\tau\)}
	\noLine
	\UnaryInfC{\(\chi, \Gamma, \bot \Rightarrow_s \Delta\)}
	\DisplayProof
	\mapsto
	\AxiomC{}
	\UnaryInfC{\(\Gamma, \bot \Rightarrow_s p, \Delta\)}
	\DisplayProof
\]
\[
	\AxiomC{\(\pi\)}
	\noLine
	\UnaryInfC{\(\Gamma, \bot \Rightarrow_s \Delta, \chi\)}
	\DisplayProof
	\hspace{0.3cm}
	\AxiomC{}
	\UnaryInfC{\(\chi, \Gamma, \bot \Rightarrow_s \Delta\)}
	\DisplayProof
	\mapsto
	\AxiomC{}
	\UnaryInfC{\(\Gamma, \bot \Rightarrow_s p, \Delta\)}
	\DisplayProof
\]
\textbf{Axiomatic character is in the cut formula}
\[
	\AxiomC{}
	\UnaryInfC{\(\Gamma, p \Rightarrow_s \Delta, p\)}
	\DisplayProof
	\hspace{0.3cm}
	\AxiomC{\(\tau\)}
	\noLine
	\UnaryInfC{\(p, \Gamma, p \Rightarrow_s \Delta\)}
	\DisplayProof
	\mapsto
	\AxiomC{\(\text{lctr}_p(\tau)\)}
	\noLine
	\UnaryInfC{\(\Gamma, p \Rightarrow_s \Delta\)}
	\DisplayProof
\]
\[
	\AxiomC{\(\pi\)}
	\noLine
	\UnaryInfC{\(\Gamma \Rightarrow_s p, \Delta, p\)}
	\DisplayProof
	\hspace{0.3cm}
	\AxiomC{}
	\UnaryInfC{\(p, \Gamma \Rightarrow_s p, \Delta\)}
	\DisplayProof
	\mapsto
	\AxiomC{\(\text{rctr}_p(\pi)\)}
	\noLine
	\UnaryInfC{\(\Gamma \Rightarrow_s p, \Delta\)}
	\DisplayProof
\]
\[
	\AxiomC{\(\pi\)}
	\noLine
	\UnaryInfC{\(\Gamma \Rightarrow_s \Delta, \bot\)}
	\DisplayProof
	\hspace{0.3cm}
	\AxiomC{}
	\UnaryInfC{\(\bot, \Gamma \Rightarrow_s \Delta\)}
	\DisplayProof
	\mapsto
	\AxiomC{\(\text{inv}_\bot(\pi)\)}
	\noLine
	\UnaryInfC{\(\Gamma \Rightarrow_s \Delta\)}
	\DisplayProof
\]

From now own we assume that neither \(\pi\) nor \(\tau\) is axiomatic, so \(\pi\) and \(\tau\) have each a principal formula.

\subsection{Principal reduction}
Assume the cut formula is principal in \(\pi\) and \(\tau\).
Then \(\chi = \chi_0 \to \chi_1\) and the cut reduction has shape

\[ 
	\AxiomC{\(\pi_0\)}
	\noLine
	\UnaryInfC{\(\chi_0, \Gamma \Rightarrow_s \Delta, \chi_1\)}
	\RightLabel{\(\to\)R}
	\UnaryInfC{\(\Gamma \Rightarrow_s \Delta, \chi_0 \to \chi_1\)}
	\DisplayProof
	\hspace{0.3cm}
	\AxiomC{\(\tau_0\)}
	\noLine
	\UnaryInfC{\(\Gamma \Rightarrow_s  \Delta, \chi_0\)}
	\AxiomC{\(\tau_1\)}
	\noLine
	\UnaryInfC{\(\chi_1, \Gamma\Rightarrow_s \Delta\)}
	\RightLabel{\(\to\)L}
	\BinaryInfC{\(\chi_0 \to \chi_1, \Gamma \Rightarrow_s \Delta\)}
	\DisplayProof
\]
\[
	\rotatebox[origin=c]{270}{$\longmapsto$}
\]
\[ 
	\AxiomC{\(\textsf{wk}(\tau_0)\)}
	\noLine
	\UnaryInfC{\(\Gamma \Rightarrow_s \Delta, \chi_1, \chi_0\)}
	\AxiomC{\(\pi_0\)}
	\noLine
	\UnaryInfC{\(\chi_0, \Gamma \Rightarrow_s \Delta, \chi_1\)}
	\RightLabel{\(\textsf{cut}_1\)}
	\BinaryInfC{\(\Gamma \Rightarrow_s \Delta, \chi_1\)}
	\AxiomC{\(\tau_1\)}
	\noLine
	\UnaryInfC{\(\chi_1, \Gamma\Rightarrow_s \Delta\)}
	\RightLabel{\(\textsf{cut}_2\)}
	\BinaryInfC{\(\Gamma \Rightarrow_s \Delta\)}
	\DisplayProof
\]

\subsection{Commutative implication cases}\label{sec:cut-reduction-com-impl}
Assume the cut formula is not principal in \(\pi\) or \(\tau\) and the principal
formula where the cut formula is not principal is an implication.
The cut reduction has 4 possible shapes:

\[ 
\AxiomC{\(\pi_0\)}
\noLine
\UnaryInfC{\(\Gamma, \phi \Rightarrow_s \psi, \Delta, \chi\)}
\RightLabel{\(\to\)R}
\UnaryInfC{\(\Gamma \Rightarrow_s \phi \to \psi, \Delta, \chi\)}
\DisplayProof
\hspace{0.3cm}
\AxiomC{\(\tau\)}
\noLine
\UnaryInfC{\(\chi, \Gamma \Rightarrow_s \phi \to \psi, \Delta\)}
\DisplayProof
\]
\[
	\rotatebox[origin=c]{270}{$\longmapsto$}
\]
\[ 
\AxiomC{\(\pi_0\)}
\noLine
\UnaryInfC{\(\Gamma, \phi \Rightarrow_s \psi, \Delta, \chi\)}
\AxiomC{\(\text{rinv}_{\phi \to \psi}(\tau)\)}
\noLine
\UnaryInfC{\(\chi, \Gamma, \phi \Rightarrow_s \psi, \Delta\)}
\RightLabel{\(\textsf{cut}_1\)}
\BinaryInfC{\(\Gamma, \phi \Rightarrow_s \psi, \Delta\)}
\RightLabel{\(\to R\)}
\UnaryInfC{\(\Gamma \Rightarrow_s \phi \to \psi, \Delta\)}
\DisplayProof
\]

\[ 
\AxiomC{\(\pi_0\)}
\noLine
\UnaryInfC{\(\Gamma \Rightarrow_s \Delta, \phi, \chi\)}
\AxiomC{\(\pi_1\)}
\noLine
\UnaryInfC{\(\Gamma, \psi \Rightarrow_s \Delta, \chi\)}
\RightLabel{\(\to\)L}
\BinaryInfC{\(\Gamma, \phi \to \psi \Rightarrow_s \Delta, \chi\)}
\DisplayProof
\hspace{0.3cm}
\AxiomC{\(\tau\)}
\noLine
\UnaryInfC{\(\chi, \Gamma, \phi \to \psi \Rightarrow_s \Delta\)}
\DisplayProof
\]
\[
	\rotatebox[origin=c]{270}{$\longmapsto$}
\]
\[ 
	\AxiomC{\(\pi_0\)}
	\noLine
	\UnaryInfC{\(\Gamma \Rightarrow_s \Delta, \phi, \chi\)}
	\AxiomC{\(\textsf{inv}^0_{\phi \to \psi}(\tau)\)}
	\noLine
	\UnaryInfC{\(\chi, \Gamma \Rightarrow_s \Delta, \phi\)}
	\RightLabel{\(\textsf{cut}_1\)}
	\BinaryInfC{\(\Gamma \Rightarrow_s \Delta, \phi\)}
	\AxiomC{\(\pi_1\)}
	\noLine
	\UnaryInfC{\(\Gamma, \psi \Rightarrow_s \Delta, \chi\)}
	\AxiomC{\(\textsf{inv}^1_{\phi \to \psi}(\tau)\)}
	\noLine
	\UnaryInfC{\(\chi, \Gamma, \psi \Rightarrow_s \Delta\)}
	\RightLabel{\(\textsf{cut}_2\)}
	\BinaryInfC{\(\Gamma, \psi \Rightarrow_s \Delta\)}
	\RightLabel{\(\to\)L}
	\BinaryInfC{\(\Gamma, \phi \to \psi \Rightarrow_s \Delta\)}
	\DisplayProof
\]

\[ 
\AxiomC{\(\pi\)}
\noLine
\UnaryInfC{\(\Gamma \Rightarrow_s \phi \to \psi, \Delta, \chi\)}
\DisplayProof
\hspace{0.3cm}
\AxiomC{\(\tau_0\)}
\noLine
\UnaryInfC{\(\chi, \Gamma, \phi \Rightarrow_s \psi, \Delta\)}
\RightLabel{\(\to\)R}
\UnaryInfC{\(\chi, \Gamma \Rightarrow_s \phi \to \psi, \Delta\)}
\DisplayProof
\]
\[
	\rotatebox[origin=c]{270}{$\longmapsto$}
\]
\[ 
\AxiomC{\(\text{rinv}_{\phi \to \psi}(\pi)\)}
\noLine
\UnaryInfC{\(\Gamma, \phi \Rightarrow_s \psi, \Delta, \chi\)}
\AxiomC{\(\tau_0\)}
\noLine
\UnaryInfC{\(\chi, \Gamma, \phi \Rightarrow_s \psi, \Delta\)}
\RightLabel{\(\textsf{cut}_1\)}
\BinaryInfC{\(\Gamma, \phi \Rightarrow_s \psi, \Delta\)}
\RightLabel{\(\to R\)}
\UnaryInfC{\(\Gamma \Rightarrow_s \phi \to \psi, \Delta\)}
\DisplayProof
\]

\[ 
\AxiomC{\(\pi\)}
\noLine
\UnaryInfC{\(\Gamma, \phi \to \psi \Rightarrow_s \Delta, \chi\)}
\DisplayProof
\hspace{0.3cm}
\AxiomC{\(\tau_0\)}
\noLine
\UnaryInfC{\(\chi, \Gamma \Rightarrow_s \Delta, \phi\)}
\AxiomC{\(\tau_1\)}
\noLine
\UnaryInfC{\(\chi, \Gamma, \psi \Rightarrow_s \Delta\)}
\RightLabel{\(\to\)L}
\BinaryInfC{\(\chi, \Gamma, \phi \to \psi \Rightarrow_s \Delta\)}
\DisplayProof
\]
\[
	\rotatebox[origin=c]{270}{$\longmapsto$}
\]
\[ 
	\AxiomC{\(\textsf{inv}^0_{\phi \to \psi}(\pi)\)}
	\noLine
	\UnaryInfC{\(\Gamma \Rightarrow_s \Delta, \phi, \chi\)}
	\AxiomC{\(\tau_0\)}
	\noLine
	\UnaryInfC{\(\chi, \Gamma \Rightarrow_s \Delta, \phi\)}
	\RightLabel{\(\textsf{cut}_1\)}
	\BinaryInfC{\(\Gamma \Rightarrow_s \Delta, \phi\)}
	\AxiomC{\(\textsf{inv}^1_{\phi \to \psi}(\pi)\)}
	\noLine
	\UnaryInfC{\(\Gamma, \psi \Rightarrow_s \Delta, \chi\)}
	\AxiomC{\(\tau_1\)}
	\noLine
	\UnaryInfC{\(\chi, \Gamma, \psi \Rightarrow_s \Delta\)}
	\RightLabel{\(\textsf{cut}_2\)}
	\BinaryInfC{\(\Gamma, \psi \Rightarrow_s \Delta\)}
	\RightLabel{\(\to\)L}
	\BinaryInfC{\(\Gamma, \phi \to \psi \Rightarrow_s \Delta\)}
	\DisplayProof  
\]

\subsection{Modal cases}
Assume the cut formula is not principal in \(\pi\) or \(\tau\) and
where the cut formula is not principal, the principal formula is of shape
\(\nec \phi\) or \(\necm \phi\).
So where the cut formula is not principal the last rule is either
\(\nec, \necm_f\) or \(\necm_u\).

\subsubsection{Weakening}\label{sec:cut-reduction-weakening}
If the cut formula belongs to the weakening part where the cut-formula is not principal, then it is trivial to eliminate the cut: just keep that proof changing the weakening part to not include the cut-formula.

\subsubsection{Commutative}
We can assume that the cut formula is not principal with last rule either \(\nec, \necm_f\) or \(\necm_u\) occurs at \(\tau\) since if it occurs at \(\pi\) it would be in the weakening part.
In addition, to not occur at the weakening part of the LHS of \(\tau\) the cut formula must be of shape \(\nec \chi_0\) or \(\necm \chi_0\).
In \(\pi\) we have two options, either the cut-formula is principal or not.

If it were not principal, then the principal formula of \(\pi\) is of shape \(\phi \to \psi\), \(\nec \phi\) or \(\necm \phi\).
The first case was already covered in Subsection~\ref{sec:cut-reduction-com-impl} and the other two will provoke that the cut formula belongs to the weakening part of a modal rule, which we covered in Subsubsection~\ref{sec:cut-reduction-weakening}.
So we can safely assume that \(\chi\) is principal in \(\pi\) and in addition, since if \(\chi = \necm \chi_0\) then \(s \neq \chi_0\) (by Subsection~\ref{sec:cut-reduction-label}), we know that the last rule applied at \(\pi\) is either \(\nec\) or \(\necm_u\).
This leaves six cases depending on the last rule of \(\pi\) and the last rule of \(\tau\).

In all the subsequent cases we are going to have \(\Sigma_0, \Sigma_1, \Gamma_0, \Gamma_1, \Pi_0, \Pi_1\) such that
\[\Sigma_0, \nec \Gamma_0, \necm \Pi_0 = \Sigma_1, \nec \Gamma_1, \necm \Pi_1.\]
Then we define \(\Gamma_2 = \Gamma_0, (\Gamma_1 \setminus \Gamma_0) = \Gamma_1, (\Gamma_0 \setminus \Gamma_1)\),  \(\Pi_2 = \Pi_0, (\Pi_1 \setminus \Pi_0) = \Pi_1, (\Pi_0 \setminus \Pi_1)\) and \(\Sigma_2\) as the only multiset such that \(\Sigma_2, \nec \Gamma_2, \necm \Pi_2 = \Sigma_0, \nec \Gamma_0, \necm \Pi_0 = \Sigma_1, \nec \Gamma_1, \necm \Pi_1\).
Notice that \(\Gamma_0, \Gamma_1 \subseteq \Gamma_2\), \(\Pi_0, \Pi_1 \subseteq \Pi_2\).

\begin{landscape}
\(\nec\)-\(\nec\) reduction.
\[ 
\AxiomC{}
\UnaryInfC{\(\pi_0 \vdash \Gamma_0, \dnecm \Pi_0 \Rightarrow_\circ \chi_0\)}
\RightLabel{\(\nec\)}
\UnaryInfC{\(\Sigma_0, \nec \Gamma_0, \necm \Pi_0 \Rightarrow_s \nec \phi, \Delta, \nec \chi_0\)}
\DisplayProof
\hspace{0.5cm}
\AxiomC{}
\UnaryInfC{\(\tau_0 \vdash \chi_0, \Gamma_1, \dnecm \Pi_1 \Rightarrow_\circ \phi\)}
\RightLabel{\(\nec\)}
\UnaryInfC{\(\nec \chi_0, \Sigma_1, \nec \Gamma_1, \necm \Pi_1 \Rightarrow_s \nec \phi, \Delta\)}
\DisplayProof
\]
\[
	\rotatebox[origin=c]{270}{$\longmapsto$}
\]
\[ 
	\AxiomC{}
	\UnaryInfC{\(\textsf{cut}_1(\textsf{wk}(\pi_0), \textsf{wk}(\tau_0)) \vdash \Gamma_2, \dnecm \Pi_2 \Rightarrow_\circ \phi\)}
	\RightLabel{\(\nec\)}
	\UnaryInfC{\(\Sigma_2, \nec \Gamma_2, \necm \Pi_2 \Rightarrow_s \nec \phi, \Delta\)}
	\DisplayProof
\]

\(\nec\)-\(\necm_u\) reduction.
\[ 
\AxiomC{}
\UnaryInfC{\(\pi_0 \vdash \Gamma_0, \dnecm \Pi_0 \Rightarrow_\circ \chi_0\)}
\RightLabel{\(\nec\)}
\UnaryInfC{\(\Sigma_0, \nec \Gamma_0, \necm \Pi_0 \Rightarrow_s \necm \phi, \Delta, \nec \chi_0\)}
\DisplayProof
\hspace{0.5cm}
\AxiomC{}
\UnaryInfC{\(\tau_0 \vdash \chi_0, \Gamma_1, \dnecm \Pi_1 \Rightarrow_\circ \phi\)}
\AxiomC{}
\UnaryInfC{\(\tau_1 \vdash \chi_0, \Gamma_1, \dnecm \Pi_1 \Rightarrow_\phi \necm\phi\)}
\RightLabel{\(\necm_u\)}
\BinaryInfC{\(\nec \chi_0, \Sigma_1, \nec \Gamma_1, \necm \Pi_1 \Rightarrow_s \necm \phi, \Delta\)}
\DisplayProof
\]
\[
	\rotatebox[origin=c]{270}{$\longmapsto$}
\]
\[ 
	\AxiomC{}
	\UnaryInfC{\(\textsf{cut}_1(\textsf{wk}(\pi_0), \textsf{wk}(\tau_0)) \vdash \Gamma_2, \dnecm \Pi_2 \Rightarrow_\circ \phi\)}
	\AxiomC{}
	\UnaryInfC{\(\textsf{cut}_2({\textsf{wk}(\pi_0)}^\phi, \textsf{wk}(\tau_1)) \vdash \Gamma_2, \dnecm \Pi_2 \Rightarrow_\circ \necm \phi\)}
	\RightLabel{\(\necm_u\)}
	\BinaryInfC{\(\Sigma_2, \nec \Gamma_2, \necm \Pi_2 \Rightarrow_s \necm \phi, \Delta\)}
	\DisplayProof
\]

\(\nec\)-\(\necm_f\) reduction.
\[ 
\AxiomC{}
\UnaryInfC{\(\pi_0 \vdash \Gamma_0, \dnecm \Pi_0 \Rightarrow_\circ \chi_0\)}
\RightLabel{\(\nec\)}
\UnaryInfC{\(\Sigma_0, \nec \Gamma_0, \necm \Pi_0 \Rightarrow_s \necm \phi, \Delta, \nec \chi_0\)}
\DisplayProof
\hspace{0.5cm}
\AxiomC{}
\UnaryInfC{\(\tau_0 \vdash \chi_0, \Gamma_1, \dnecm \Pi_1 \Rightarrow_\circ \phi\)}
\AxiomC{\(\tau_1\)}
\noLine
\UnaryInfC{\(\chi_0, \Gamma_1, \dnecm \Pi_1 \Rightarrow_\phi \necm\phi\)}
\RightLabel{\(\necm_u\)}
\BinaryInfC{\(\nec \chi_0, \Sigma_1, \nec \Gamma_1, \necm \Pi_1 \Rightarrow_s \necm \phi, \Delta\)}
\DisplayProof
\]
\[
	\rotatebox[origin=c]{270}{$\longmapsto$}
\]
\[ 
	\AxiomC{}
	\UnaryInfC{\(\textsf{cut}_1(\textsf{wk}(\pi_0), \textsf{wk}(\tau_0)) \vdash \Gamma_2, \dnecm \Pi_2 \Rightarrow_\circ \phi\)}
	\AxiomC{\({\textsf{wk}(\pi_0)}^\phi\)}
	\noLine
	\UnaryInfC{\(\Gamma_2, \dnecm \Pi_2 \Rightarrow_\phi \necm \phi, \chi_0\)}
	\AxiomC{\(\textsf{wk}(\tau_1)\)}
	\noLine
	\UnaryInfC{\(\chi_0, \Gamma_2, \dnecm \Pi_2 \Rightarrow_\phi \necm \phi\)}
	\RightLabel{\(\textsf{cut}_2\)}
	\BinaryInfC{\(\Gamma_2, \dnecm \Pi_2 \Rightarrow_\phi \necm \phi\)}
	\RightLabel{\(\necm_f\)}
	\BinaryInfC{\(\Sigma_2, \nec \Gamma_2, \necm \Pi_2 \Rightarrow_s \necm \phi, \Delta\)}
	\DisplayProof
\]

\newpage
\(\necm_u\)-\(\nec\) reduction.
\[ 
\AxiomC{}
\UnaryInfC{\(\pi_0 \vdash \Gamma_0, \dnecm \Pi_0 \Rightarrow_\circ \chi_0\)}
\AxiomC{}
\UnaryInfC{\(\pi_1 \vdash \Gamma_0, \dnecm \Pi_0 \Rightarrow_\phi \necm\chi_0\)}
\RightLabel{\(\necm_u\)}
\BinaryInfC{\(\Sigma_0, \nec \Gamma_0, \necm \Pi_0 \Rightarrow_s \nec \phi, \Delta, \necm \chi_0\)}
\DisplayProof
\hspace{0.3cm}
	\AxiomC{}
	\UnaryInfC{\(\tau_0 \vdash \dnecm \chi_0, \Gamma_1, \dnecm \Pi_1 \Rightarrow_\circ \phi\)}
	\RightLabel{\(\nec\)}
	\UnaryInfC{\(\necm \chi_0, \Sigma_1, \nec \Gamma_1, \necm \Pi_1 \Rightarrow_s \nec \phi, \Delta\)}
	\DisplayProof
\]
\[
	\rotatebox[origin=c]{270}{$\longmapsto$}
\]
\[ 
\AxiomC{}
\UnaryInfC{\(\rho_0 \vdash \Gamma_2, \dnecm \Pi_2 \Rightarrow_\circ \phi\)}
\RightLabel{\(\nec\)}
\UnaryInfC{\(\nec \Gamma_2, \necm \Pi_2 \Rightarrow_\circ \nec \phi, \Delta\)}
\DisplayProof
\]
where \(\rho_0 = \textsf{cut}_2(\textsf{wk}(\pi_0),\textsf{cut}_1({\textsf{wk}(\pi_1)}^{\circ}, \textsf{wk}(\tau_0)))\).

\(\necm_u\)-\(\necm_u\) reduction.
{
\footnotesize
\[ 
\AxiomC{}
\UnaryInfC{\(\pi_0 \vdash \Gamma_0, \dnecm \Pi_0 \Rightarrow_\circ \chi_0\)}
\AxiomC{}
\UnaryInfC{\(\pi_1 \vdash \Gamma_0, \dnecm \Pi_0 \Rightarrow_{\chi_0} \necm\chi_0\)}
\RightLabel{\(\necm_u\)}
\BinaryInfC{\(\Sigma_0, \nec \Gamma_0, \necm \Pi_0 \Rightarrow_s \necm \phi, \Delta, \necm \chi_0\)}
\DisplayProof
\hspace{0.3cm}
	\AxiomC{}
	\UnaryInfC{\(\tau_0 \vdash \dnecm \chi_0, \Gamma_1, \dnecm \Pi_1 \Rightarrow_\circ \phi\)}
	\AxiomC{}
	\UnaryInfC{\(\tau_1 \vdash \dnecm \chi_0, \Gamma_1, \dnecm \Pi_1 \Rightarrow_\phi \necm \phi\)}
	\RightLabel{\(\necm_u\)}
	\BinaryInfC{\(\necm \chi_0, \Sigma_1, \nec \Gamma_1, \necm \Pi_1 \Rightarrow_s \necm \phi, \Delta\)}
	\DisplayProof
\]}
\[
	\rotatebox[origin=c]{270}{$\longmapsto$}
\]
\[ 
	\AxiomC{}
	\UnaryInfC{\(\rho_0 \vdash \Gamma_2, \dnecm \Pi_2 \Rightarrow_\circ \phi\)}
	\AxiomC{}
	\UnaryInfC{\(\rho_1 \vdash \Gamma_2, \dnecm \Pi_2 \Rightarrow_\phi \necm\phi\)}
	\RightLabel{\(\necm_u\)}
	\BinaryInfC{\(\Sigma_2, \nec \Gamma_2, \necm \Pi_2 \Rightarrow_s \necm \phi\)}
	\DisplayProof
\]
where \(\rho_0 = \textsf{cut}_2(\textsf{wk}(\pi_0),\textsf{cut}_1({\textsf{wk}(\pi_1)}^{\circ}, \textsf{wk}(\tau_0)))\) and \(\rho_1 = \textsf{cut}_4({\textsf{wk}(\pi_0)}^\phi,\textsf{cut}_3({\textsf{wk}(\pi_1)}^{\phi}, \textsf{wk}(\tau_1)))\)

\newpage
\(\necm_u\)-\(\necm_f\) reduction.
{
\footnotesize
\[ 
\AxiomC{}
\UnaryInfC{\(\pi_0 \vdash \Gamma_0, \dnecm \Pi_0 \Rightarrow_\circ \chi_0\)}
\AxiomC{}
\UnaryInfC{\(\pi_1 \vdash \Gamma_0, \dnecm \Pi_0 \Rightarrow_{\chi_0} \necm\chi_0\)}
\RightLabel{\(\necm_u\)}
\BinaryInfC{\(\Sigma_0, \nec \Gamma_0, \necm \Pi_0 \Rightarrow_s \necm \phi, \Delta, \necm \chi_0\)}
\DisplayProof
\hspace{0.3cm}
	\AxiomC{}
	\UnaryInfC{\(\tau_0 \vdash \dnecm \chi_0, \Gamma_1, \dnecm \Pi_1 \Rightarrow_\circ \phi\)}
	\AxiomC{\(\tau_1\)}
	\noLine
	\UnaryInfC{\(\dnecm \chi_0, \Gamma_1, \dnecm \Pi_1 \Rightarrow_\phi \necm \phi\)}
	\RightLabel{\(\necm_f\)}
	\BinaryInfC{\(\necm \chi_0, \Sigma_1, \nec \Gamma_1, \necm \Pi_1 \Rightarrow_s \necm \phi, \Delta\)}
	\DisplayProof
\]}
\[
	\rotatebox[origin=c]{270}{$\longmapsto$}
\]
{\footnotesize
\[ 
	\AxiomC{}
	\UnaryInfC{\(\rho_0 \vdash \Gamma_2, \dnecm \Pi_2 \Rightarrow_\circ \phi\)}
	\AxiomC{\({\textsf{wk}(\pi_0)}^\phi\)}
	\noLine
	\UnaryInfC{\(\Gamma_2, \dnecm \Pi_2 \Rightarrow_\phi \necm \phi, \chi_0\)}
	\AxiomC{\({\textsf{wk}(\pi_1)}^{\phi}\)}
	\noLine
	\UnaryInfC{\(\chi_0, \Gamma_2, \dnecm \Pi_2 \Rightarrow_\phi \necm \phi, \necm \chi_0\)}
	\AxiomC{\(\textsf{wk}(\tau_1)\)}
	\noLine
	\UnaryInfC{\(\dnecm \chi_0, \Gamma_2, \dnecm \Pi_2 \Rightarrow_\phi \necm \phi\)}
	\RightLabel{\(\textsf{cut}_3\)}
	\BinaryInfC{\(\chi_0, \Gamma_2, \dnecm \Pi_2 \Rightarrow_\phi \necm \phi\)}
	\RightLabel{\(\textsf{cut}_4\)}
	\BinaryInfC{\(\Gamma_2, \dnecm \Pi_2 \Rightarrow_\phi \necm \phi\)}
	\RightLabel{\(\necm_f\)}
	\BinaryInfC{\(\Sigma_2, \nec \Gamma_2, \necm \Pi_2 \Rightarrow_s \necm \phi\)}
	\DisplayProof
\]
}
where \(\rho_0 = \textsf{cut}_2(\textsf{wk}(\pi_0),\textsf{cut}_1({\textsf{wk}(\pi_1)}^{\circ}, \textsf{wk}(\tau_0)))\).
\end{landscape}

\printbibliography

\end{document}